\documentclass{amsart}
\usepackage{amssymb,latexsym}
\usepackage[usenames]{color}
\usepackage{lscape}
\usepackage{url}
\usepackage[OT2,T1]{fontenc}

\newtheorem{theorem}{Theorem}[section]
\newtheorem{corollary}[theorem]{Corollary}
\newtheorem{lemma}{Lemma}[section]
\newtheorem{prop}[theorem]{Proposition}
\theoremstyle{remark}

\newcommand{\QQ}{{\mathbb{Q}}} 
\newcommand{\Qbar}{{\overline{\QQ}}} 

\DeclareMathOperator{\Gal}{Gal}

\def\mod#1{{\ifmmode\text{\rm\ (mod~$#1$)}
\else\discretionary{}{}{\hbox{ }}\rm(mod~$#1$)\fi}}

\newcommand{\F}{\mathbb{F}}

\newcommand{\Q}{\mathbb{Q}}

\newcommand{\rhobar}{{\overline{\rho}}}

\newcommand{\fp}{\mathfrak{p}}

\newcommand{\legendre}[2]{\genfrac{(}{)}{}{}{#1}{#2}}

\numberwithin{equation}{section}

\theoremstyle{definition}

\def\Qbar{{\overline{\QQ}}} 

\def\mod{\mathop{\rm mod}}

\def\mod#1{{\ifmmode\text{\rm\ (mod~$#1$)}
\else\discretionary{}{}{\hbox{ }}\rm(mod~$#1$)\fi}}

\def\Java{\hbox{\sc J\kern -1pt{}ava}}

\begin{document}

\title{Sums of two cubes as twisted perfect powers, revisited}

\author{Michael A. Bennett}
\email{bennett@math.ubc.ca}
\address{Department of Mathematics, University of British Columbia, Vancouver, B.C., V6T 1Z2 Canada}
\thanks{The first author was supported in part by a grant from NSERC}
\urladdr{http://www.math.ubc.ca/$\sim$bennett/}

\author{Carmen Bruni}
\email{cbruni@uwaterloo.ca}
\address{Centre for Education in Mathematics and Computing, University of Waterloo, Waterloo, Ontario, N2L 3G1 Canada}
\urladdr{http://www.cemc.uwaterloo.ca/$\sim$cbruni/}

\author{Nuno Freitas}
\email{nuno@math.ubc.ca}
\address{Department of Mathematics, University of British Columbia, Vancouver, B.C., V6T 1Z2 Canada}
\thanks{The third author was supported in part by the grant {\it Proyecto RSME-FBBVA $2015$ Jos\'e Luis Rubio de Francia}}
\urladdr{http://www.math.ubc.ca/$\sim$nuno/}

\begin{abstract}

In this paper, we sharpen earlier work of the first author, Luca and Mulholland \cite{BLM}, showing that the Diophantine equation 
$$
A^3+B^3 = q^\alpha C^p, \; \; ABC \neq 0, \; \; \gcd (A,B) =1,
$$
has, for ``most'' primes $q$ and suitably large prime exponents $p$, no solutions. We  handle a number of (presumably infinite) families where no such conclusion was hitherto known. Through further 
application of certain {\it symplectic criteria}, we are able to make some conditional statements about still more values of $q$; a sample such result is that, for all but $O(\sqrt{x}/\log x)$ primes $q$ up to $x$, the equation
$$
 A^3 + B^3 = q C^p.
$$
has no solutions in  coprime, nonzero integers $A, B$ and $C$, for a positive proportion of prime exponents $p$.
 
\end{abstract}

\maketitle

\section{Introduction}

The problem of classifying perfect powers that are representable as a sum of two coprime integer cubes has a long history.  The nonexistence of cubes $C^3 > 1$ with this property, a special case of Fermat's Last Theorem, was essentially proven by Euler. For higher powers, we have a substantial amount of recent work; at the time of writing, this can be summarized in the following theorem.

\begin{theorem}[Bruin \cite{Br00}, Chen-Siksek \cite{ChSi}, Dahmen \cite{Dah}, Freitas \cite{Fr}, Kraus \cite{Kr}]
There are no solutions in relatively prime nonzero  integers $A, B$ and $C$ to the equation
\begin{equation} \label{start}
A^3+B^3=C^n
\end{equation}
with exponent $n$ satisfying one of
$$
\begin{array}{c} 
3 \leq n\leq 10^9, \;  n \equiv 2 \mod 3,  \; n \equiv 2,3 \mod 5,\\
n \equiv 61 \mod {78}, \; n \equiv 51,103,105 \mod {106} \mbox{, or}\\
    n \equiv 43,49,61,79,97,151,157,169,187,205,259,265,277, \\
   295,313,367,373,385,403,421,475,481,493,511,529, \\
   583, 601,619,637, 691,697,709,727,745,799,805,817,835, \\
     853,907, 913,925,943,961, 1015,1021,1033,1051,1069,1123, \\
     1129,1141, 1159,1177, 1231,1237,1249,1267,1285 \mod {1296}. \\ 
   \end{array}
$$

\end{theorem}

\vskip2ex
Underlying each of these results is an appeal to a particular Frey-Hellegouarch elliptic curve, defined over $\mathbb{Q}$. Just as in the case of Fermat's Last Theorem, with analogous equation $A^n+B^n=C^n$, this curve corresponds to a particular weight $2$, cuspidal newform $f$. In the latter case, Wiles \cite{Wi} showed that $f$ necessarily has level $2$ (whereby the absence of such newforms implies an immediate contradiction). In the case of equation (\ref{start}), however, one finds a corresponding $f$ at one of levels $18, 36$ or $72$. The first two of these are readily handled, but the last is not. The obstruction to completely resolving equation (\ref{start}) is the existence of a particular elliptic curve over $\mathbb{Q}$ with conductor $72$ which, on some level, ``mimics'' a solution to (\ref{start}) (the curve in question is labelled  $72A1$ in
Cremona's tables \cite{Cre})

In an earlier paper \cite{BLM}, the first author, joint with Luca and Mulholland, considered a modification of equation (\ref{start}), where the right-hand-side is replaced by a ``twisted'' version of the shape $q^\alpha C^p$, for $q$ prime (the replacement of the exponent $n$ with a prime one $p$ is without loss of generality). The question the authors of \cite{BLM} wished to answer was whether or not a similar obstruction exists in this new situation. 
Here and henceforth, let us assume that we have a  proper, nontrivial solution
$(a,b,c)$ of the equation 
\begin{equation} \label{ABC}
A^3+B^3= q^\alpha C^p,
\end{equation}
i.e. a solution with $A, B$ and $C$ nonzero, coprime integers and $\alpha$ a positive integer. 
Write $S$ for the set of primes $q \geq 5$ for which there exists an elliptic curve $E/\mathbb{Q}$ with conductor $N (E) \in \{ 18 q, 36 q, 72 q \}$ and at least one nontrivial rational $2$-torsion point. The two main results of \cite{BLM} are the following :

\begin{theorem}[B., Luca, Mulholland \cite{BLM}]  \label{BLM1}
If $p$ and $q \geq 5$ are primes with $p \geq q^{2q}$ such that there exist coprime, nonzero integers $A, B$ and $C$, and a positive integer $\alpha$, satisfying equation $\eqref{ABC}$, then $q \in S$.
\end{theorem}

\begin{theorem}[B., Luca, Mulholland \cite{BLM}] \label{fudge2}
 Let  $\pi_S (x) = \# \{ q \leq x \; : \; q \in S \}$. Then
\begin{equation} \label{prig}
\pi_S (x) \ll \sqrt{x} \log^2 (x).
\end{equation}
\end{theorem}
This latter result may be reasonably easily sharpened, through sieve methods, but, even as stated, demonstrates that $\pi_S(x) = o (\pi (x))$ and hence that we may ``solve'' equation (\ref{ABC}) for ``almost all'' primes $q$ (i.e. for almost all primes, there is no analogous obstruction to that provided by the curve $72A1$ for equation (\ref{start})).

Our goal in the paper at hand is to improve this result by treating equation (\ref{ABC}) for a significant number of the primes in $S$. We begin by defining $S_0$ to be the subset of $S$ consisting of those primes $q \geq 5$ for which there exists an elliptic curve $E/\mathbb{Q}$ with conductor $N (E) \in \{ 18 q, 36 q, 72 q \}$, nontrivial rational $2$-torsion and the additional property that discriminant $\Delta (E) = T^2$ or $\Delta (E) = -3T^2$ for some integer $T$.
The first main result of this paper is the following sharpening of Theorem \ref{BLM1}.

\begin{theorem} \label{main1} 
If $p$ and $q \geq 5$ are primes with $p \geq q^{2q}$ such that there exist coprime, nonzero integers $A, B$ and $C$, and a positive integer $\alpha$, satisfying equation $\eqref{ABC}$, then $q \in S_0$.
\end{theorem}

It is by no means clear that the set $S_0$ is appreciably ``smaller'' than $S$. In fact, our expectation is that their counting functions satisfy
$$
\pi_S (x) \sim c_1 \sqrt{x} \log x \; \; \mbox{ and } \; \; \pi_{S_0} (x) \sim c_2 \sqrt{x} \log x,
$$
for positive constants $c_1$ and $c_2$, where $c_2 < c_1$.
A cursory check of Cremona's elliptic curve database \cite{Cre} reveals that
the primes $5 \leq q < 1000$ lying outside $S$ are precisely 
$$
q=197, 317, 439, 557, 653, 677, 701, 773, 797 \mbox{ and } 821,
$$
while, in the same range,
the primes in $S$ but not $S_0$ are
$$
\begin{array}{c}
q = 53, 83, 149, 167, 173, 199, 223, 227, 233,  263, 281, 293, 311, 347, 353, 359, 389,  \\
401, 419, 443, 449, 461, 467, 479, 487, 491, 563, 569, 571, 587, 599, 617, 641, 643,  \\
659, 719, 727, 739, 743, 751, 809, 811, 823, 827, 829, 839, 859, 881, 887, 907, 911,  \\
929, 941, 947, 953, 977 \mbox{ and }  983. \\
\end{array}
$$
It is, in fact, possible to give a much more concrete characterization of $S_0$.
Let us define sets
$$
S_1 = \left\{ q \mbox{ prime }  :  q=2^a 3^b \pm 1, \; \; a \in \{ 2, 3 \} \mbox{ or } a \geq 5, b \geq 0 \right\},
$$
$$
S_2 = \left\{ q \mbox{ prime }  :  q=\left| 2^a  \pm 3^b \right|, \; \; a \in \{ 2, 3 \} \mbox{ or } a \geq 5, b \geq 1 \right\},
$$
$$
S_3 = \left\{ q \mbox{ prime }  :  q=\frac{2^a +1}{3}, \; \; a \geq 5 \mbox{ odd } \right\},
$$
$$
S_4 = \left\{ q \mbox{ prime }  :  q=d^2+2^a 3^b, \; \; a \in \{ 2, 4 \} \mbox{ or } a \geq 8 \mbox{ even, } b \mbox{ odd} \right\},
$$
$$
S_5 = \left\{ q \mbox{ prime }  :  q=3 d^2+2^a, \; \; a \in \{ 2, 4 \} \mbox{ or } a \geq 8 \mbox{ even }  \right\},
$$
$$
S_6 = \left\{ q \mbox{ prime }  :  q=\frac{d^2+3^b}{4}, \; \; d \mbox{ and } b \mbox{ odd} \right\},
$$
$$
S_7 = \left\{ q \mbox{ prime }  :  q=\frac{3d^2 +1}{4}, \; \; d \mbox{ odd }  \right\}
$$
and
$$
S_8 = \left\{ q \mbox{ prime }  :  q=\frac{3v^2 -1}{2},  \; \; u^2-3v^2=-2  \right\}.
$$
Here, $a, b, u, v$ and $d$ are integers.

\begin{prop}
We have
$$
S_0 = S_1 \cup S_2 \cup S_3  \cup S_4 \cup S_5  \cup S_6 \cup S_7  \cup S_8.
$$
\end{prop}
An advantage of this characterization is that it makes it a routine matter to check if a given prime is in $S_0$ (something that is far from being true for $S$). 
It also allows one to rather easily find, via local conditions, sets of primes outside $S_0$; simply checking that $S_0$ contains no primes which are simultaneously $5 \mod{8}$, $2 \mod{3}$ and $3 \mod{5}$, yields  that if $q \equiv 53 \mod{120}$, then $q \not\in S_0$. More generally, from Theorem \ref{main1}, we deduce the following.

\begin{corollary} \label{cor-blimey} 
If $p$ and $q$ are primes with either $q \equiv 53 \mod{D_1}$ for $D_1 \in \{ 96, 120, 144 \}$ or $q \equiv 65 \mod{D_2}$ for $D_2 \in \{  81, 84 \}$, and $p \geq q^{2q}$, then there are no coprime, nonzero integers $A, B$ and $C$, and positive integers $\alpha$, satisfying equation $\eqref{ABC}$.
\end{corollary}

For primes in $S_0$, we are often still able to say something about solutions to (\ref{ABC}), in many cases eliminating a positive proportion of the possible prime exponents $p$. Indeed, let us define 
$$
T = S_7 \cup \left\{  q \mbox{ prime }  :  q=3 d^2+16, \; \; d \in \mathbb{Z}  \right\},
$$
and, to simplify matters, suppose that $\alpha = 1$ in (\ref{ABC}),  focussing our attention on the equation
\begin{equation} \label{gump}
A^3 + B^3 = q \, C^p.
\end{equation}
We have the following.
\begin{theorem} \label{shark99}
If $q$ is a prime with $q \not\in T$, then, for a positive proportion of primes $p$, there are no solutions to equation \eqref{gump} in coprime nonzero integers $A, B$ and $C$.
\end{theorem}
We note that, defining $\pi_T(x)$ to be the counting function for primes in $T$, it is not difficult to show that 
$$
\pi_T (x) \ll \frac{\sqrt{x}}{ \log x},
$$
whereby standard heuristics suggest that the set $T$ is genuinely of smaller order than $S_0$ (though, in point of fact, it would be remarkably difficult to prove that either set is even infinite).

As a sampling of more explicit work along these lines, we mention the following results for certain primes in $S_0$  (see also Theorem~\ref{rump} in Section~\ref{applications}). 
\begin{theorem} \label{main3} 
Suppose that  $q=2^a 3^b-1$ is prime, where $a \geq 5$ and $b \geq 1$ integers. If $p > q^{2q}$ is prime and there exists a positive integer $\alpha$ and coprime, nonzero integers $A, B$ and $C$ satisfying equation $\eqref{ABC}$, then
$$
\left( \frac{\alpha}{p} \right) =  \left( \frac{4-a}{p} \right) =  \left( \frac{-6b}{p} \right).
$$
\end{theorem}

\begin{theorem} \label{main4}
 If $p$ is prime  with $p \equiv 13, 19$ or $23 \mod{24}$, then there are no coprime, nonzero integers $A, B$ and $C$ satisfying
 \begin{equation} \label{quiet}
 A^3 + B^3 = 5 C^p.
 \end{equation}
 \end{theorem}
 
 These results all follow from applying the modular method, together with a somewhat elaborate blend of  techniques from algebraic and analytic number theory, and Diophantine approximation, with a variety of {\it symplectic criteria}  (see Section \ref{apply}) to equation (\ref{ABC}). This last approach was developed initially by 
Halberstadt and Kraus \cite{HK2002} and has recently been refined and generalized
in the work of the third author \cite{Fr}, jointly 
with Naskr{\k e}cki and Stoll \cite{FNS23n} and with Kraus \cite{FrKr1}.
One of the justifications for the current paper is to provide a number of examples which, on some level,  utilize the full power of these recently developed symplectic tools.  

As a final comment, we note that it should be possible to apply techniques based upon quadratic reciprocity, as in, say, work of Chen and Siksek \cite{ChSi}, to say something further about equation (\ref{ABC}) for certain primes $q$ and certain exponents. We will not undertake this here.

\smallskip
  
 The outline of this paper is as follows. In Section \ref{Frey}, we restate a number of results from \cite{BLM} pertaining to Frey-Hellegouarch curves that we require in the sequel.
 In Section \ref{Class}, we characterize isomorphism classes of elliptic curves over $\mathbb{Q}$ with nontrivial rational $2$-torsion and conductor $18q, 36q$ or $72q$, for $q$ prime. 
 Section \ref{proofie} contains the proof of Theorem \ref{main1}. 
In  Section \ref{freebie}, we make a number of remarks about the sets $S_i$ comprising $S_0$. 
In Section \ref{apply}, we apply several symplectic criteria to the Frey-Hellegouarch curve and
the elliptic curves corresponding to the primes in $S_0$.
In Section \ref{applications}, we prove Theorems \ref{shark99}, \ref{main3} and \ref{main4} (and somewhat more besides). 
Section \ref{appendix} is an appendix containing information on the invariants $c_4(E)$ and $c_6(E)$  for elliptic curves of conductor $18q, 36q$ and $72q$, corresponding to the primes in $S$.

\section{Frey-Hellegouarch curves} \label{Frey}

Let us suppose that $q \geq 5$ is prime, $\alpha$ is a positive integer, and that we have a solution to equation (\ref{ABC}) in coprime nonzero integers $A, B$ and $C$ where, without loss of generality, $AC$ is even and $B \equiv (-1)^{C+1} \mod{4}$.
Following Darmon and Granville \cite{DG}, we associate to such a solution a {\it Frey-Hellegouarch elliptic curve} $F = F^{(i)}_{A,B}$ of the shape
$$
F^{(0)}_{A,B} \; \; : \; \; y^2+xy=x^3 + \frac{3( B-A) +2}{8} x^2 + \frac{3(A+B)^2}{64} x + \frac{ 9 (B-A)(A+B)^2}{512}
$$
or
$$
 F^{(1)}_{A,B} \; \; : \; \;  y^2=x^3+3ABx+B^3-A^3,
$$
depending on whether $C$ is even or odd, respectively. For future reference, we note that these are minimal models.
 The standard invariants $c_4(F), c_6(F)$ and
 $\Delta(F)$ attached to $F=F^{(i)}_{A,B}$ are 
 \begin{equation}\label{eqn:invariants}
 \begin{cases}
  c_4(F)=-2^{4i} 3^2AB,\\
  c_6(F)=2^{6i-1} 3^3(A^3-B^3),\\
  \Delta(F)=-2^{12i-8} 3^3 q^{2\alpha} C^{2p}.
 \end{cases}
\end{equation}

Let $\mathcal{R}$ denote the product of the primes $\ell$ satisfying $ \ell \mid C$ and $\ell \nmid 6q$.
A standard application of Tate's algorithm leads to the following.
\begin{lemma}
\label{lem:conductor-of-frey-curve-2}
If $F=F^{(i)}_{A,B}$, then the  conductor $N_F$ satisfies
 $$
 N_{F}=\begin{cases}
               18 \, q\, \mathcal{R}&\text{ if $C$ even, }B\equiv -1\text{ {\rm (mod} } 4), \; 
               \text{or}\\
               36 \, q\, \mathcal{R}&\text{ if $C$ odd, } v_2(A)\geq 2\text{ and }B\equiv1\text{ {\rm (mod} } 4), \; \text{or}\\
               72 \, q \, \mathcal{R}&\text{ if $C$ odd, } v_2(A)=1\text{ and }B\equiv1\text{ {\rm (mod} } 4).
              \end{cases}
$$
In particular, $F$ has multiplicative reduction at the prime $q$. 
\end{lemma}

Arguing as in \cite{BLM} and \cite{Kr} we find that, for $p \geq 17$, there necessarily exists a newform $f\in
S_2^+(N_{F}/\mathcal{R})$ (the space of weight $2$ cuspidal newforms for the congruence subgroup $\Gamma_0(N_{F}/\mathcal{R})$), whose Taylor expansion is
$$
f= q + \sum_{m \ge 2}a_m(f)q^m \text{ where } q=e^{2\pi i t},
$$
and a place $\fp$ of $\Qbar$ lying above
$p$, such that 
\begin{equation} \label{iso1}
\rhobar_{F,p} \sim \rhobar_{f,\fp},
\end{equation}
where $\rhobar_{F,p}$ and $\rhobar_{f,\fp}$ denote, respectively, the
mod~$\fp$ Galois representations attached to $F$ and $f$.
In particular,
for all prime numbers $\ell \nmid pN_{F}$, we have
$$
a_{\ell}(f) \equiv a_{\ell}(F) \pmod{\fp},
$$
where $a_{\ell}(F)$ denotes the trace of Frobenius of $F$ at the prime $\ell$. Therefore,
\begin{align}
\label{eqn:n-big-bound}
 p \mid {\rm Norm}_{K_{f}/\QQ}(a_{\ell}(f)-a_{\ell}(F)),
\end{align}
for $K_f$ the field of definition of the coefficients of $f$. 
Furthermore, the level lowering condition implies 
\begin{equation} \label{bad}
 p \mid {\rm Norm}_{K_{f}/\QQ}(a_{\ell}(f) \pm (\ell+1)),
 \end{equation}
for each prime $\ell \neq p$ dividing $\mathcal{R}$. 

From the arguments of \cite{BLM}, under the assumption that $p > q^{2p}$, we may conclude that the form $f$ has rational integer Fourier coefficients
$a_m(f)$ for all $m\ge1$, whereby $f$ corresponds to an isogeny class of
elliptic curves over $\QQ$ with conductor $N=18q$, $36q$ or
$72q$, and further that the corresponding elliptic curve $E$ has a rational $2$-torsion point. This, in essence, is Theorem \ref{BLM1}. To complete the proof of Theorem~\ref{main1}, it remains to eliminate the possibility of the Frey-Hellegouarch curve $F$ ``arising mod~$p$'' from an elliptic curve $E$ that fails to be isogenous to a curve with discriminant of the shape $T^2$ or $-3T^2$. To do this, we first require a very precise characterization of elliptic curves of conductor $N=18q$, $36q$ or
$72q$, with nontrivial rational $2$-torsion.

\section{Classification results for primes of conductor $18q$, $36q$ and $72q$} \label{Class}

In this section, we will state theorems that provide an explicit classification for primes $q$ of the corresponding isomorphism classes of elliptic curves $E/\mathbb{Q}$ with conductor $18q$, $36q$ or $72q$ and nontrivial rational $2$-torsion. 
The following results are mild sharpenings and simplifications of special cases of Theorems 3.13, 3.14 and 3.15  of Mulholland \cite{Mu} (see also Theorems 4.0.8, 4.0.10 and 4.0.12 of \cite{Bruni}), where analogous results are derived more generally for elliptic curves with nontrivial rational $2$-torsion and conductor of the shape $2^\alpha 3^\beta q^\gamma$.

\begin{theorem} \label{theorem18q}
If $q > 3$ is prime, then there exists an elliptic curve $E/\mathbb{Q}$ of conductor $18q$ with at least one rational $2$-torsion point precisely when either $E$ is isogenous to one of (in Cremona's notation) 
$$
\begin{array}{c}
90a, 90b, 90c, 126a, 126b, 198b, 198c, 198d, 198e, 306a,  \\
 306b, 306c, 342c, 342f, 414a, 1314a \mbox{ or } 1314f, \\
\end{array} 
$$
or $E$ is $\mathbb{Q}$-isomorphic to
$$
\tilde{E} \; \; : \; \; y^2 +  xy  = x^3 + a_2 x^2 + a_4 x + a_6
$$ 
and at least one of the following occurs:
\begin{enumerate}
\item There exist integers $a \geq 5$ and $b \geq 0$ such that
$$
q=2^{a} 3^b + (-1)^\delta, \; \; \mbox{ for } \; \; \delta \in \{ 0, 1 \},
$$
and $\tilde{E}$ is one of the following :
$$
\begin{array}{|c|c|c|} \hline
\mbox{curve} & a_2 & a_4  \\ \hline
18q.1.a1 & (-1)^{\delta+1} 2^{a-1} 3^{b+1} - 1 & 2^{a-4} 3^{b+2} q   \\
18q.1.a2 & (-1)^{\delta+1} 2^{a-1} 3^{b+1} - 1 &  -2^{a-2} 3^{b+2} q    \\
18q.1.a3 & (-1)^{\delta+1}2^{a-3} 3^{b+1} - 1 & 2^{2a-8} 3^{2b+2}   \\
18q.1.a4 & (-1)^{\delta}2^{a-2} 3^{b+1} - 1 & (-1)^{\delta}2^{a-2} 3^{b+2} \\ \hline
\mbox{curve} & a_6 & \Delta \\ \hline
18q.1.a1 & 0 & 2^{2a-8} 3^{2b+6} q^2 \\
18q.1.a2 & (-1)^{\delta} 2^{a-4} 3^{b+3} \left( 2q- (-1)^{\delta} \right) q & 2^{a-4} 3^{b+6} q \\
18q.1.a3 & 0 & (-1)^{\delta} 2^{4a-16} 3^{4b+6} q \\
18q.1.a4 & 2^{a-4} 3^{b+3} ( q - 2 (-1)^{\delta})  & (-1)^{\delta+1} 2^{a-4} 3^{b+6} q^4 \\
\hline
\end{array}
$$

\item There exists an odd  integer $a \geq 5$ such that 
$$
q=\frac{2^{a} + 1}{3}
$$
and $\tilde{E}$ is one of the following :
 $$
\begin{array}{|c|c|c|} \hline
\mbox{curve} & a_2 & a_4  \\ \hline
18q.2.a1 &-3 \cdot 2^{a-1}  -1 &  2^{a-4} 3^{3} q   \\
18q.2.a2 &  -3 \cdot 2^{a-1}  -1 &   -2^{a-2} 3^{3} q     \\
18q.2.a3 & -3 \cdot 2^{a-3}  -1 & 2^{2a-8} 3^{2}   \\
18q.2.a4 & 3 \cdot 2^{a-2}  -1 & 2^{a-2} 3^{2} \\ \hline
\mbox{curve} & a_6 & \Delta \\ \hline
18q.2.a1 & 0 & 2^{2a-8} 3^{8} q^2 \\
18q.2.a2 & 2^{a-4} 3^{4} \left( 2^{a+1} + 1 \right) q & 2^{a-4} 3^{7} q \\
18q.2.a3 & 0 &  2^{4a-16} 3^{7} q \\
18q.2.a4 & 2^{a-4} 3^{3} \left( 2^{a} - 1 \right)   & -2^{a-4} 3^{10} q^4 \\
\hline
\end{array}
$$

\item There exist integers $a \geq 5$ and  $b \geq 1$, and $\delta_1, \delta_2 \in \{ 0, 1 \}$ such that
$$
q = (-1)^{\delta_1} 2^a + (-1)^{\delta_2}  3^b
$$
and, writing $\delta = b+\delta_1+\delta_2+1$, $\tilde{E}$ is one of the following :
$$
\begin{array}{|c|c|c|} \hline
\mbox{curve} & a_2 & a_4  \\ \hline
18q.3.a1 & - \frac{1}{4} + (-1)^{\delta} \left( 3 \cdot 2^{a-2}  + (-1)^{\delta_1} \frac{3q}{4}) \right) &  (-1)^{\delta_1} 2^{a-4} 3^{2} q   \\
18q.3.a2 &  - \frac{1}{4} + (-1)^{\delta} \left( 3 \cdot 2^{a-2}  + (-1)^{\delta_1} \frac{3q}{4}) \right) &   (-1)^{\delta_1+1}2^{a-2} 3^{2} q    \\
18q.3.a3 & - \frac{1}{4} +  (-1)^{\delta}  3 \cdot 2^{a-3}  - (-1)^b \frac{3^{b+1}}{4}  & 2^{2a-8} 3^{2}  \\
18q.3.a4 & - \frac{1}{4} + (-1)^{\delta+1} \left( 3 \cdot 2^{a-1}  + (-1)^{\delta_1+1} \frac{3q}{4}) \right) & (-1)^{\delta_1+\delta_2}2^{a-2} 3^{b+2} \\ \hline
\mbox{curve} & a_6 & \Delta \\ \hline
18q.3.a1 & 0 & 2^{2a-8} 3^{2b+6} q^2  \\
18q.3.a2 & (-1)^{\delta+\delta_1+1} 3^3 \cdot 2^{a-4} (2^{a+1}+  (-1)^{\delta_1+\delta_2} 3^b) q & 2^{a-4} 3^{4b+6} q  \\
18q.3.a3 & 0 &  (-1)^{\delta_2} 2^{4a-16} 3^{b+6} q \\
18q.3.a4 & (-1)^{\delta} 3^{b+3} \cdot 2^{a-4} (3^b+ (-1)^{1+\delta_1+\delta_2} 2^a)   & (-1)^{b+\delta} 2^{a-4} 3^{b+6} q^4  \\
\hline
\end{array}
$$

\item There exist integers $a \geq 7$, $b \geq 0$, $\delta_1, \delta_2 \in \{ 0, 1 \}$ and $d \equiv 1 \pmod{4}$, such that $(\delta_1, \delta_2) \neq (1,1)$ and, if we have $a \equiv b \equiv 0 \pmod{2}$, then $(\delta_1,\delta_2)=(0,0)$, with
$$
q = (-1)^{\delta_1} d^2+ (-1)^{\delta_2} 2^a 3^b,
$$
and $\tilde{E}$ is one of the following :
$$
\begin{array}{|c|c|c|} \hline
\mbox{curve} & a_2 & a_4  \\ \hline
18q.4.a1 & - \left( \frac{3d+1}{4} \right) & (-1)^{\delta_1+\delta_2+1} 2^{a-6} 3^{b+2}   \\
18q.4.a2 & - \left( \frac{3d+1}{4} \right) &  (-1)^{\delta_1+\delta_2} 2^{a-4} 3^{b+2}    \\ \hline
\mbox{curve} & a_6 & \Delta \\ \hline
18q.4.a1 & 0 &  (-1)^{\delta_1} 2^{2a-12} 3^{2b+6} q  \\
18q.4.a2 & (-1)^{\delta_1+\delta_2+1} 2^{a-6} 3^{b+3} d & (-1)^{\delta_1+\delta_2+1} 2^{a-6} 3^{b+6} q^2  \\
\hline
\end{array}
$$

\item There exist integers $a \geq 7$, $\delta_1, \delta_2 \in \{ 0, 1 \}$ and $d \equiv 1 \pmod{4}$, such that $(\delta_1, \delta_2) \neq (1,1)$,
$$
q = (-1)^{\delta_1} 3 d^2+ (-1)^{\delta_2} 2^a,
$$
and $\tilde{E}$ is one of the following :
$$
\begin{array}{|c|c|c|} \hline
\mbox{curve} & a_2 & a_4  \\ \hline
18q.5.a1 &  - \left( \frac{3d+1}{4} \right) &  (-1)^{\delta_1+\delta_2+1} 2^{a-6} 3   \\
18q.5.a2 &   - \left( \frac{3d+1}{4} \right) &  (-1)^{\delta_1+\delta_2} 2^{a-4} 3    \\
18q.5.b1 & \frac{9d-1}{4} &  (-1)^{\delta_1+\delta_2+1} 2^{a-6} 3^3  \\
18q.5.b2 & \frac{9d-1}{4} & (-1)^{\delta_1+\delta_2} 2^{a-4} 3^3 \\ \hline
\mbox{curve} & a_6 & \Delta \\ \hline
18q.5.a1 & 0 &  (-1)^{\delta_1} 2^{2a-12} 3^{3} q \\
18q.5.a2 &  (-1)^{\delta_1+\delta_2+1} 2^{a-6} 3^{2} d &  (-1)^{\delta_1+\delta_2+1} 2^{a-6} 3^{3} q^2  \\
18q.5.b1 & 0 & (-1)^{\delta_1} 2^{2a-12} 3^{9} q \\
18q.5.b2 & (-1)^{\delta_1+\delta_2} 2^{a-6} 3^{5} d  & (-1)^{\delta_1+\delta_2+1} 2^{a-6} 3^{9} q^2  \\
\hline
\end{array}
$$

\item There exist integers $a \geq 7$, $b \geq 1$ and $d \equiv 1 \pmod{4}$, with $a$ odd, such that 
$$
q =\frac{d^2+ 2^a}{3^b},
$$
and $\tilde{E}$ is one of the following :
\begin{small}
$$
\begin{array}{|c|c|c|c|c|} \hline
\mbox{curve} & a_2 & a_4  & a_6 & \Delta \\ \hline
18q.6.a1 & - \left( \frac{3d+1}{4} \right) & - 2^{a-6} 3^{2} & 0  &  2^{2a-12} 3^{b+6} q \\
18q.6.a2 & - \left( \frac{3d+1}{4} \right) & 2^{a-4} 3^{2} & -2^{a-6} 3^{3} d &  -2^{a-6} 3^{2b+6} q^2  \\
\hline
\end{array}
$$
\end{small}

\end{enumerate}
\end{theorem}

\vskip2ex

\begin{theorem} \label{theorem36q}
If $q > 3$ is prime, then there exists an elliptic curve $E/\mathbb{Q}$ of conductor $36q$ with at least one rational $2$-torsion point precisely when either $E$ is isogenous to one of (in Cremona's notation) 
$$
180a, 252a \mbox{ or } 468d,
$$
or $E$ is $\mathbb{Q}$-isomorphic to
$$
\tilde{E} \; \; : \; \; y^2   = x^3 + a_2 x^2 + a_4 x 
$$ 
and at least one of the following occurs:
\begin{enumerate}
\item There exist  integers $u$ and $v$ with $u \equiv v \equiv 1 \pmod{4}$ and $u^2-3v^2=-2$, such that
$$
q = \frac{3v^2-1}{2}
$$
and $\tilde{E}$ is one of the following :
$$
\begin{array}{|c|c|c|c|} \hline
\mbox{curve} & a_2 & a_4 & \Delta \\ \hline
36q.1.a1 & -3uv & 3q^2 &  -2^{4} 3^{3} q^4 \\
36q.1.a2 & 6 uv & -3 &  2^{8} 3^{3} q^2 \\ \hline
36q.1.b1 & 9 uv & 3^3  q^2 &  -2^{4} 3^{9} q^4 \\
36q.1.b2 & -18 uv & -3^3 &  2^{8} 3^{9} q^2 \\
\hline
\end{array}
$$

\item There exists an  integer $d \equiv 1 \mod{8}$, such that
$$
q = \frac{3d^2+1}{4}
$$
and $\tilde{E}$ is one of the following :
$$
\begin{array}{|c|c|c|c|} \hline
\mbox{curve} & a_2 & a_4 & \Delta \\ \hline
36q.2.a1 & -3d & 3q & -2^{4} 3^{3} q^2 \\
36q.2.a2 & 6 d & -3  & 2^{8} 3^{3} q \\ \hline
36q.2.b1 & 9 d & 3^3 q & -2^{4} 3^{9} q^2 \\
36q.2.b2 & -18 d & -3^3 &   2^{8} 3^{9} q \\
\hline
\end{array}
$$

\item  There exists  an odd integer $b \geq 1$ and an integer $d \equiv 1 \pmod{4}$ such that
$$
q = \frac{d^2+3^b}{4} \equiv 3 \pmod{4}
$$
and $\tilde{E}$ is one of the following :
$$
\begin{array}{|c|c|c|c|} \hline
\mbox{curve} & a_2 & a_4 & \Delta  \\ \hline
36q.3.a1 & -3d & 3^2 q & -2^{4} 3^{b+6} q^2  \\
36q.3.a2 & 6d & -3^{b+2}  & 2^{8} 3^{2b+6} q \\
\hline
\end{array}
$$

\item There exist integers $b \geq 1$, $\delta \in \{ 0, 1 \}$  and $d \equiv 1 \pmod{4}$, such that $b$ is odd, $d \equiv 1 \mod{4}$,
$$
q = (-1)^{\delta} \left( d^2- 4 \cdot 3^{b} \right)
$$
and $\tilde{E}$ is one of the following :
$$
\begin{array}{|c|c|c|c|} \hline
\mbox{curve} & a_2 & a_4 & \Delta  \\ \hline
36q.4.a1 & -3d & 3^{b+2} &  (-1)^{\delta} 2^{4} 3^{2b+6} q \\
36q.4.a2 & 6d & (-1)^{\delta} 3^2 q & 2^{8} 3^{b+6} q^2 \\
\hline
\end{array}
$$

\item There exist integers $b \geq 1$, $\delta \in \{ 0, 1 \}$, $n \geq 7$  and $d \equiv 1 \pmod{4}$, such that $b$ is odd, every prime factor 
of $n$ is at least $7$,
$$
q^n = (-1)^{\delta} \left( d^2 - 4 \cdot 3^{b} \right)
$$
and $\tilde{E}$ is one of the following :
$$
\begin{array}{|c|c|c|c|} \hline
\mbox{curve} & a_2 & a_4 & \Delta \\ \hline
36q.5.a1 & -3d & 3^{b+2} &  (-1)^{\delta} 2^{4} 3^{2b+6} q^n \\
36q.5.a2 & 6d & (-1)^{\delta} 3^2 q^n &  2^{8} 3^{b+6} q^{2n} \\
\hline
\end{array}
$$

\item There exists an  integer $d \equiv 1 \mod{4}$, such that
$$
q = 3 d^2-4
$$
and $\tilde{E}$ is one of the following :
$$
\begin{array}{|c|c|c|c|} \hline
\mbox{curve} & a_2 & a_4 & \Delta  \\ \hline
36q.6.a1 & -3d & 3 & 2^{4} 3^{3} q \\
36q.6.a2 & 6 d & 3 q  & 2^{8} 3^{3} q^2 \\ \hline
36q.6.b1 & 9 d & 3^3 & 2^{4} 3^{9} q  \\
36q.6.b2 & -18 d & 3^3 q & 2^{8} 3^{9} q^2 \\
\hline
\end{array}
$$

\item There exists an  integer $d \equiv 1 \mod{4}$, and an even integer $b \geq 0$ such that
$$
q = d^2+ 4 \cdot 3^b,
$$
and $\tilde{E}$ is one of the following :
$$
\begin{array}{|c|c|c|c|} \hline
\mbox{curve} & a_2 & a_4 & \Delta  \\ \hline
36q.7.a1 & -3d & -3^{b+2} & 2^4 3^{2b+6} q \\
36q.7.a2 & 6d & 3^2 q  & -2^8 3^{b+6} q^2 \\
\hline
\end{array}
$$

\end{enumerate}
\end{theorem}

\begin{theorem} \label{theorem72q}
 If $q > 3$ is prime, then there exists an elliptic curve $E/\mathbb{Q}$ of conductor $72q$ with at least one rational $2$-torsion point precisely when either $E$ is isogenous to one of (in Cremona's notation) 
$$
\begin{array}{c}
360a, 360b, 360c, 360d, 936a, 936d, 936f,  2088b,    \\
2088h, 3384a, 5256e, 13896f  \mbox{ or } 83016c, \\
\end{array} 
$$
or $E$ is $\mathbb{Q}$-isomorphic to
$$
\tilde{E} \; \; : \; \; y^2   = x^3 + a_2 x^2 + a_4 x 
$$ 
and at least one of the following occurs:
\begin{enumerate}
\item There exists an odd  integer $b \geq 1$ such that
$$
q = \frac{3^b+1}{4}
$$
and $\tilde{E}$ is one of the following :
$$
\begin{array}{|c|c|c|c|} \hline
\mbox{curve} & a_2 & a_4 & \Delta\\ \hline
72q.1.a1 & 24 q-3 & 2^2 3^{b+2} q &  2^{8} 3^{2b+6} q^2 \\
72q.1.a2 & -48 q + 6 &  3^2 & 2^{10} 3^{b+6} q \\
72q.1.a3 & 24q+6 & 3^{2b+2}  & 2^{10} 3^{4b+6}  q \\
72q.1.a4 & 6q-3 & 3^2 q^2 & -2^{4} 3^{b+6} q^4 \\
\hline
\end{array}
$$

\item There exist integers $a \in \{ 2, 3 \}$, $b \geq 0$ and $\delta \in \{ 0, 1 \}$ such that
$$
q=2^a \cdot 3^b + (-1)^\delta
$$
and $\tilde{E}$ is one of the following :
$$
\begin{array}{|c|c|c|c|} \hline
\mbox{curve} & a_2 & a_4 & \Delta \\ \hline
72q.2.a1 & (-1)^{\delta+1} 2^{a+1} 3^{b+1} - 3 & 2^a 3^{b+2} q & 2^{2a+4} 3^{2b+6} q^2  \\
72q.2.a2 & (-1)^{\delta} 2^{a+2} 3^{b+1} +6 & 3^2 & 2^{a+8} 3^{b+6} q \\
72q.2.a3 & (-1)^{\delta+1} 2^{a-1}  3^{b+1} -3 & 2^{2a-4} 3^{2b+2} & (-1)^\delta 2^{4a-4} 3^{4b+6} q  \\
72q.2.a4 &(-1)^{\delta+1} 2^{a+1} 3^{b+1} +6 & 3^2 q^2 & (-1)^{\delta+1} 2^{a+8} 3^{b+6} q^4  \\
\hline
\end{array}
$$

\item  There exist integers $a \in \{ 2, 3 \}$, $b \geq 0$ and $\delta \in \{ 0, 1 \}$   such that 
$$
q=3^b  + (-1)^\delta 2^a
$$
and $\tilde{E}$ is one of the following :
$$
\begin{array}{|c|c|c|c|} \hline
\mbox{curve} & a_2 & a_4 & \Delta \\ \hline
72q.3.a1 & (-1)^{b+1} 3 \left( 3^b - (-1)^\delta 2^a \right)  & (-1)^{\delta+1} 2^a 3^{b+2}  & 2^{2a+4} 3^{2b+6} q^2 \\
72q.3.a2 & (-1)^{b} 6 \left( 3^b - (-1)^\delta 2^a \right)  & 3^2 q^2 &  (-1)^{\delta+1} 2^{a+8} 3^{b+6} q^4 \\
72q.3.a3 & (-1)^{b} 6 \left( 3^b + (-1)^\delta 2^{a+1} \right) & 3^{2b+2} & (-1)^{\delta} 2^{a+8} 3^{4b+6} q \\
72q.3.a4 &  (-1)^{b+1} 3 \left( 3^b + (-1)^\delta 2^{a-1} \right)  & 2^{2a-4} 3^{2} &  2^{4a-4} 3^{b+6} q \\
\hline
\end{array}
$$

\item  There exists  an integer $d \equiv 1 \pmod{4}$ such that 
$$
q=3 d^2  + 4
$$
and $\tilde{E}$ is one of the following :
$$
\begin{array}{|c|c|c|c|} \hline
\mbox{curve} & a_2 & a_4 & \Delta \\ \hline
72q.4.a1 & 3 d & -  3  & 2^{4} 3^{3} q \\
72q.4.a2 &  - 6 d &  3 q  & - 2^{8} 3^{3} q^2 \\ \hline
72q.4.b1 & - 9 d &   - 3^{3} & 2^{4} 3^{9} q  \\
72q.4.b2 & 18 d & 3^3 q  & - 2^{8} 3^{9} q^2 \\
\hline
\end{array}
$$

\item  There exist integers $a \in \{ 4, 5 \}$,  $\delta \in \{ 0, 1 \}$ and  $d \equiv 1 \pmod{4}$, such that 
$$
q=3 d^2  + (-1)^\delta 2^a
$$
and $\tilde{E}$ is one of the following :
$$
\begin{array}{|c|c|c|c|} \hline
\mbox{curve} & a_2 & a_4 & \Delta \\ \hline
72q.5.a1 & - 3 d & (-1)^{\delta+1} 2^{a-2} 3  & 2^{2a} 3^{3} q \\
72q.5.a2 &  6 d &  3 q & (-1)^{\delta+1} 2^{a+6} 3^{3} q^2 \\ \hline
72q.5.b1 &  9 d &   (-1)^{\delta+1} 2^{a-2}3^{3} & 2^{2a} 3^{9} q \\
72q.5.b2 & - 18 d & 3^3 q & (-1)^{\delta+1} 2^{a+6} 3^{9} q^2 \\
\hline
\end{array}
$$

\item  There exist an integer $d \equiv 5 \mod{8}$ such that  
$$
q=\frac{3 d^2+1}{4}
$$
and $\tilde{E}$ is one of the following :
$$
\begin{array}{|c|c|c|c|} \hline
\mbox{curve} & a_2 & a_4 & \Delta \\ \hline
72q.6.a1 &  3 d & 3 q &  -2^{4} 3^{3} q^2 \\
72q.6.a2 &  -6 d &  -3 & 2^{8} 3^{3} q \\ \hline
72q.6.b1 & -9 d &  3^3 q &  -2^{4} 3^{9} q^2 \\
72q.6.b2 & 18 d & -3^3 & 2^{8} 3^{9} q \\
\hline
\end{array}
$$

\item  There exist odd  integers $b \geq 1$ and $d \equiv 1 \mod{4}$ such that  
$$
q=\frac{d^2+3^b}{4} \equiv 1 \mod{4}
$$
and $\tilde{E}$ is one of the following :
$$
\begin{array}{|c|c|c|c|} \hline
\mbox{curve} & a_2 & a_4 & \Delta \\ \hline
72q.7.a1 &  3 d & 3^2 q  & -2^{4} 3^{b+6} q^2 \\
72q.7.a2 &  -6 d &  -3^{b+2} & 2^{8} 3^{2b+6} q \\ 
\hline
\end{array}
$$

\item  There exist odd  integers $b \geq 1$ and $d \equiv 1 \pmod{4}$ such that  
$$
q=d^2 + 4 \cdot 3^b
$$
and $\tilde{E}$ is one of the following :
$$
\begin{array}{|c|c|c|c|} \hline
\mbox{curve} & a_2 & a_4 & \Delta \\ \hline
72q.8.a1 &  3 d & -3^{b+2} &  2^{4} 3^{2b+6} q \\
72q.8.a2 &  -6 d &  3^2 q & -2^{8} 3^{b+6} q^2  \\ 
\hline
\end{array}
$$

\item There exist integers $a \in \{ 4, 5 \}$, $b \geq 0$, $\delta_1, \delta_2 \in \{ 0, 1 \}$ and $d \equiv 1 \pmod{4}$, such that $(\delta_1, \delta_2) \neq (1,1)$,
$b$ is odd if $a=4$ and $\delta_1 \neq \delta_2$,
$$
q = (-1)^{\delta_1} d^2+ (-1)^{\delta_2} 2^a 3^b,
$$
and $\tilde{E}$ is one of the following :
$$
\begin{array}{|c|c|c|c|} \hline
\mbox{curve} & a_2 & a_4& \Delta \\ \hline
72q.9.a1 &  -3 d & (-1)^{\delta_1+\delta_2+1} 2^{a-2} 3^{b+2}  & (-1)^{\delta_1} 2^{2a} 3^{2b+6} q \\
72q.9.a2 &  6 d &  (-1)^{\delta_1} 3^2 q & (-1)^{\delta_1+\delta_2+1} 2^{a+6} 3^{b+6} q^2 \\ 
\hline
\end{array}
$$

\item There exist integers $a \in \{ 4, 5 \}$, $b \geq 0$, $\delta \in \{ 0, 1 \}$,  $d \equiv 1 \mod{4}$ and $n$, such that  
the least prime divisor of $n$ is at least $7$, $b$ is odd if $a=4$,
$$
q^n = (-1)^{\delta} \left( d^2 - 2^a 3^b \right),
$$
and $\tilde{E}$ is one of the following : 
$$
\begin{array}{|c|c|c|c|} \hline
\mbox{curve} & a_2 & a_4 & \Delta \\ \hline
72q.10.a1 &  -3 d &  2^{a-2} 3^{b+2} &  (-1)^{\delta} 2^{2a} 3^{2b+6} q^n \\
72q.10.a2 &  6 d &  (-1)^{\delta} 3^2 q^n & 2^{a+6} 3^{b+6} q^{2n}  \\ 
\hline
\end{array}
$$

\item  There exist  integers $b \geq 0$ and $d \equiv 1 \mod{4}$ such that  
$$
q=\frac{d^2+32}{3^b}
$$
and $\tilde{E}$ is one of the following :
$$
\begin{array}{|c|c|c|c|} \hline
\mbox{curve} & a_2 & a_4 & \Delta \\ \hline
72q.11.a1 & - 3 d & -2^{3} 3^2 & 2^{10} 3^{b+6} q  \\
72q.11.a2 &  6 d &  3^{b+2} q & -2^{11} 3^{2b+6} q^2 \\ 
\hline
\end{array}
$$

\item  There exist  integers $b \geq 0$, $d \equiv 1 \mod{4}$ and $n$, such that  
the least prime divisor of $n$ is at least $7$, 
$$
q^n=\frac{d^2+32}{3^b}
$$
and $\tilde{E}$ is one of the following :
$$
\begin{array}{|c|c|c|c|} \hline
\mbox{curve} & a_2 & a_4 & \Delta \\ \hline
72q.12.a1 & - 3 d & -2^{3} 3^2 &  2^{10} 3^{b+6} q^n \\
72q.12.a2 &  6 d &  3^{b+2} q^n & -2^{11} 3^{2b+6} q^{2n} \\ 
\hline
\end{array}
$$

\end{enumerate}
\end{theorem}

We should mention that while we are currently unable to rule out the existence of primes in families (10) and (12) in Theorem \ref{theorem72q}, we strongly suspect that there are no such primes. Further, we must confess that our notation can admit a certain amount of ambiguity as, for a given prime $q$, we could have multiple representations of $q$ giving rise to non-isogenous curves with the same labels. By way of example, 
\begin{equation} \label{ex99}
q=10369 = 1^2 + 2^7 \cdot 3^4 = 65^2 + 2^{11} \cdot 3
\end{equation}
and the curves denoted $18q.4.a$ corresponding to these two representations are non-isogenous. For $q \in S_i$ for a fixed $1 \leq i \leq 8$, however, it is straightforward to show that there are at most finitely many such distinct such representations -- for all except $i=2$, the parametrization monotonically increasing in the variables $a,b, v$ and $d$. For $q \in S_2$, the same is easily seen to be  true except, possibly, for the cases with $q = | 2^a - 3^b|$. In this last situation, via a result of Tijdeman \cite{Tij}, we have
$$
\left| 2^a - 3^b \right| \geq 3^b b^{-\kappa},
$$
for some effectively computable absolute positive constant $\kappa$, at least provided $b > 2$, and hence, again, $q$ has only finitely many such representations (at most $3$, in fact, by a result of the first author \cite{Be2003}).

Combining  Theorems \ref{theorem18q}, \ref{theorem36q} and \ref{theorem72q}, together with the definition of $S_0$, yields the following.
\begin{corollary} \label{wonder}
An elliptic curve $E/\mathbb{Q}$ is in $S_0$ precisely if $E$ is either in one of the isogeny classes (in Cremona's notation) 
$$
90c, 126b, 252a, 306c, 342f, 360a, 360d, 936d \mbox{ or } 5256e,
$$
or $E$ is isogenous to one of
$$
\begin{array}{c}
18q.1.a, 18q.2.a, 18q.3.a, 18q.4.a \; (\mbox{with } \delta_1=\delta_2=0, a \mbox{ even}, b \mbox{ odd}), 18.q.5.a \mbox{ and }  \\
18q.5.b \; (\mbox{with, in both cases, } \delta_1=\delta_2=0 \mbox{ and } a \mbox{ even}), 36q.1.a, 36q.1.b, 36q.2.a,   \\
36q.2.b, 36q.3.a,  72q.1.a, 72q.2.a, 72q.3.a, 72q.4.a, 72q.4.b, 72.q.5.a  \mbox{ and } 72q.5.b \\
(\mbox{with, in both cases, } \delta=0 \mbox{ and } a =4), 72q.6.a, 72q.6.b, 72q.7.a, 72q.8.a  \\
 \mbox{ and }  72q.9.a \; (\mbox{with } \delta_1=\delta_2=0, a =4, \, b \mbox{ odd}).\\
\end{array}
$$
\end{corollary}

\section{Finishing the proof of Theorem \ref{main1}} \label{proofie}

From the classification results of the preceding section, we need to show only that, for suitably large primes $p$, equation (\ref{ABC}) has no solutions in coprime nonzero integers, with Frey-Hellegouarch curve $F$
corresponding (in the sense of Section \ref{Frey}) to an elliptic curve $E$ in one of the isogeny classes 
\begin{equation} \label{gander}
\begin{array}{c}
90a, 90b, 126a, 180a, 198b, 198c, 198d, 198e, 306a, 306b,  \\
 342b,  342c, 360b, 360c, 414a, 468d,  936a, 936f, 1314a, \\
 1314f, 2088b,  2088h, 3384a, 13896f  \mbox{ or } 83016c, \\
\end{array}
\end{equation}
or
\begin{equation} \label{goose}
\begin{array}{c}  
18q.4.a \; (\mbox{with } \delta_1 \neq \delta_2, \mbox{ or } \delta_1=\delta_2=0 \mbox{ and either } a \mbox{ odd, or } b \mbox{ even}),   \\
18.q.5.a  \mbox{ and } 18q.5.b \; (\mbox{with, in both cases, } \delta_1 \neq \delta_2, \mbox{ or } \delta_1=\delta_2=0  \mbox{ and }  \;   \\
a \mbox{ odd}), 18q.6.a,  36q.4.a  \; 36q.5.a, \; 36q.6.a, \;  36q.6.b, \;  36q.7.a, \; 72.q.5.a     \\ 
\mbox{and } 72q.5.b \; (\mbox{with, in both cases, } \delta=1 \mbox{ or } a =5),  \; 72q.9.a   \\
 (\mbox{with }  \delta_1 \neq \delta_2,  
 \mbox{or } \delta_1= 
\delta_2=0 \mbox{ and either } a = 5, \mbox{ or } a=4 
\mbox{ and $b$ even}), \\
   72q.10.a, \; 72q.11.a \mbox{ or } 72q.12.a. \\
\end{array}
\end{equation}
Our key observation to start is that, from (\ref{eqn:invariants}), the Frey-Hellegouarch curve $F=F_{A,B}^{(i)}$ has minimal discriminant of the shape $-3 T^2$ for $T=2^{6i-4} 3 q^{\alpha} C^{p}$. It follows that 
$E (\mathbb{F}_{\ell})$ contains a subgroup isomorphic to $\mathbb{Z}_2 \times \mathbb{Z}_2$ for every prime $\ell \nmid 6q$ for which $\left(\frac{-3}{\ell}\right)=1$; i.e. for $\ell \equiv 1 \mod{6}$.
We thus have that
\begin{equation} \label{shark}
a_{\ell} (F) \equiv \ell+1 \mod{4}
\end{equation}
for every such prime $\ell$. If, for each curve $E$ in the isogeny classes (\ref{gander}) and (\ref{goose}), we are able to find a prime $\ell \equiv 1 \mod{6}$ with $\ell \nmid 6q$, for which $a_{\ell} (E) \not\equiv \ell+1 \mod{4}$, it follows from (\ref{eqn:n-big-bound}), (\ref{bad}), (\ref{shark}) and the Hasse bounds  that
\begin{equation} \label{final}
p \leq \ell+1 + 2 \sqrt{\ell}.
\end{equation}

For curves $E$ in the isogeny classes (\ref{gander}), we may check that it suffices to choose, in all cases, 
$$
\ell \in \{ 7,  13, 19, 31, 37 \}.
$$
To show that we can always find a suitable prime $\ell$ for $E$ in the isogeny classes  (\ref{goose}), let us suppose that, more generally, we have an elliptic curve $E$ over $\mathbb{Q}$, with a rational $2$-torsion point at, say, $(0,0)$, given by the model
\begin{equation} \label{scooter}
E \; \; : \; \; y^2 = x^3 + ux^2 +vx.
\end{equation}
For a prime of good reduction $\ell \geq 5$ for $E$, it follows that the Fourier coefficient $a_{\ell} (E)$ satisfies $a_{\ell} (E) \equiv \ell+1 \mod{4}$ 
precisely when $E (\mathbb{F}_{\ell})$ contains a subgroup isomorphic to either $\mathbb{Z}_2 \times \mathbb{Z}_2$ or to $\mathbb{Z}_4$. The first case occurs exactly when the cubic $x^3 + ux^2 +vx$ splits completely modulo $\ell$, i.e. when
$$ 
\left( \frac{\Delta (E)}{\ell} \right) = \left( \frac{u^2 - 4v}{\ell} \right) = 1.
$$
To have the second case, there must exist a point $P$ in $E (\mathbb{F}_{\ell})$ with the property that $P \neq (0,0)$ but $2P = (0,0)$. From the standard duplication formula, the $x$-coordinate of the point $2P$ on $E$ satisfies
$$
x(2P) = \frac{x^{4}-2v x^{2} + v^{2}}{4y^{2}} = \frac{(x^2 - v)^2}{4y^{2}}
$$
and hence there can exist a point $P$ with $x(2P)=0$ precisely when $v$ is a square modulo $\ell$. In summary, we have that $a_{\ell} (E) \equiv \ell+1 \mod{4}$  exactly when either
$$
\left( \frac{\Delta (E)}{\ell} \right) = 1 \; \; \mbox{ or } \; \; \left( \frac{v}{\ell} \right) = 1.
$$

For the curves of conductor $36q$ and $72q$ ($36q.4.a$, $36q.5.a$, $36q.6.a$, $36q.6.b$, 
 $72.q.4.a$, $72q.4.b$, $72q.8.a$, $72q.9.a$, $72q.10.a$ and $72q.11.a.$), our given models are already of the form $y^2 = x^3 + ux^2 +vx$, with $u=a_2(E)$ and $v=a_4(E)$. For our families of conductor $18q$ ($18q.4.a$, $18q.5.a$, $18q.5.b$ and $18q.6.a$), we need to move our nontrivial rational $2$-torsion point to $(0,0)$ to obtain a (non-minimal) model of the shape (\ref{scooter}) (the discriminant remaining invariant modulo squares). We summarize our results in the following table.

$$
\begin{array}{|c|c|c|}
\hline
E & \mbox{ Additional conditions } & \left\{ \legendre{v}{\ell}, \legendre{\Delta (E)}{\ell} \right\} \\ [\smallskipamount]
\hline 
18q.4.a  & \mbox{unless } \delta_1 = \delta_2 = 0, \; a \mbox{ even and }  b \mbox{ odd  } &  \left\{ \legendre{(-1)^{\delta_1} q}{\ell}, \legendre{(-1)^{\delta_1+\delta_2+1} 2^a 3^b}{\ell}   \right\}   \\
18q.5.a  \mbox{ and } b & \mbox{unless } \delta_1 = \delta_2 = 0, \; a  \mbox{ even }  &  \left\{ \legendre{(-1)^{\delta_1} 3 q}{\ell}, \legendre{(-1)^{\delta_1+\delta_2+1} 2^a 3}{\ell}   \right\}   \\
18q.6.a  & \mbox{ none } &  \left\{ \legendre{3^b q}{\ell}, \legendre{-2}{\ell}   \right\}   \\
36q.4.a  & \mbox{ none }   & \left\{ \legendre{(-1)^{\delta_1}  q}{\ell}, \legendre{3 }{\ell} \right\} \\
36q.5.a  & \mbox{none}  & \left\{ \legendre{3}{\ell}, \legendre{(-1)^\delta q}{\ell} \right\}  \\
36q.6.a \mbox{ and } b  & \mbox{none}  & \left\{ \legendre{3 q}{\ell}, \legendre{ 3}{\ell}  \right\} \\
36q.7.a  & \mbox{ none }   & \left\{ \legendre{-1}{\ell}, \legendre{q }{\ell} \right\} \\
72q.5.a \mbox{ and } b  & (\delta,a) = (0,5), (1,4) \mbox{ or } (1,5) & \left\{ \legendre{3 q}{\ell}, \legendre{(-1)^{\delta+1} 2^{a} 3}{\ell}  \right\} \\
72q.9.a  & \mbox{unless } \delta_1 = \delta_2 = 0, \; a =4 \mbox{ and }  b \mbox{ odd  } &  \left\{ \legendre{(-1)^{\delta_1} q}{\ell}, \legendre{(-1)^{\delta_1+\delta_2+1} 2^a 3^b}{\ell}   \right\}   \\
72q.10.a  & \mbox{none} &  \left\{  \legendre{(-1)^\delta q}{\ell}, \legendre{2^a 3^b}{\ell}   \right\}   \\
72q.11.a \mbox{ and } 12a & \mbox{none} & \left\{  \legendre{3^bq}{\ell}, \legendre{-2}{\ell}   \right\}  \\
\hline
\end{array}
$$

For example, in case $E=72q.12.a1$, we have, for $\ell \nmid 6q$,
$$
\left( \frac{\Delta (E)}{\ell} \right) = \left( \frac{3^b q}{\ell} \right)   \; \; \mbox{ and } \; \;   \left( \frac{v}{\ell} \right)  = \left( \frac{a_4(E)}{\ell} \right) =  \left( \frac{-2}{\ell} \right).
$$
In particular, if we assume, say, that $b$ is odd, for any prime $\ell \equiv 7 \mod{24}$ such that $q$ is a quadratic residue modulo $\ell$, or prime $\ell \equiv 13 \mod{24}$ with $q$ a quadratic non-residue modulo $\ell$,  we have
\begin{equation} \label{boat}
\left( \frac{\Delta (E)}{\ell} \right) =  \left( \frac{v}{\ell} \right) = -1
\end{equation}
and hence for such a prime $\ell$, both $\ell \equiv 1 \mod{6}$ and  
\begin{equation} \label{shark1}
a_{\ell} (E) \not \equiv \ell+1 \mod{4}.
\end{equation}
In both cases, we therefore obtain inequality (\ref{final}). For the other isogeny classes in the above table, in each case there exists at least one pair of integers $(\ell_0,\delta)$, with $\ell_0 \in \{ 1, 7, 13, 19 \}$  
and $\delta \in \{ 0, 1 \}$, such that if 
\begin{equation} \label{cookie}
\ell \equiv \ell_0 \mod{24} \; \; \mbox{ and } \; \; \left( \frac{q}{\ell} \right) = (-1)^\delta, 
\end{equation}
then (\ref{boat})  and (\ref{shark1}) hold. To complete the proof of Theorem \ref{main1}, from (\ref{final}), we require a suitably strong upper bound for the smallest $\ell$ satisfying (\ref{cookie}). Such a bound would follow from either a modified version of the arguments traditionally used to find smallest non-residues modulo $q$ (though the additional constraint that $\ell \equiv \ell_0 \mod{24}$ causes some complications), or from an explicit version of Linnik's theorem on the smallest prime in a given arithmetic progression (see e.g. \cite{HB} for an effective but inexplicit result along these lines). For our purposes (and since we require something completely explicit), we will instead appeal to a recent result of the first author, Martin, O'Bryant and Rechnitzer \cite{BMOR}; here, $\theta (x; k, a)$ denotes the sum of the logarithms of the primes $p \equiv a \mod{k}$ with $p \leq x$.

\begin{theorem}[B.,  Martin, O'Bryant and Rechnitzer \cite{BMOR}] \label{main theta theorem} 
Let  $k$ and $a$ be integers with $k \ge 3$ and $\gcd (a,k) =1$. Then
$$
 \left| \theta (x; k, a) - \frac{x}{\phi (k)} \right| < \frac{0.00164 \, k}{\phi (k)} \, \frac{x}{\log x},
$$
for all $x \geq x_0(k)$, where
\begin{equation} \label{x_0(k) definition}
x_0(k) = \begin{cases}
7.5\times10^7, &\text{if } 3\le k \le 266, \\
2\times10^{10}/k, &\text{if } 267\le k \le 10^5, \\
\exp(0.0442 \sqrt k\log^3k), &\text{if } k>10^5.
\end{cases}
\end{equation}
\end{theorem}
With this result in hand, we may prove the following
\begin{prop} \label{chumpy}
Let $q \geq 7$ be prime and suppose that $\ell_0 \in \{ 1, 7, 13, 19 \}$  
and $ \delta \in \{ 0, 1 \}$. Then there exists a prime $\ell \neq q$ satisfying $\eqref{cookie}$ with 
$\ell < e^q$.
\end{prop}

\begin{proof}
Given $\ell_0 \in \{ 1, 7, 13, 19 \}$  and $ \delta \in \{ 0, 1 \}$, conditions (\ref{cookie}) are equivalent, via the Chinese Remainder Theorem,  to a congruence of the shape $\ell \equiv a \mod{24 q}$ for some integer $7 \leq a < 24q$ with $\gcd (a,24q)=1$. For $7 \leq q \leq 17$ and each of the $8$ pairs $(\ell_0, \delta)$, we verify  by direct computation that we can always find an $\ell < e^q$ with (\ref{cookie})
(this fails to be true for $q=5$ and $\ell_0=\delta=1$, which is why we have omitted this value). If $q \geq 19$, then $e^q > x_0(24q)$ and hence we may apply Theorem \ref{main theta theorem} to conclude that 
$$
 \left| \theta (e^q; 24 q, a) - \frac{e^q}{4 (q-1)} \right| < \frac{0.00984 \, e^q}{q-1},
$$
whereby
$$
 \theta (e^q; 24 q, a) > \frac{0.24 \, e^q}{q-1} > \log q.
 $$
 It follows that there exists a prime $\ell \equiv a \mod{24q}$ (which necessarily  also satisfies (\ref{cookie})) with $\ell \neq q$ and $\ell < e^q$, as desired.
\end{proof}

If $q=5$, then, for $\ell_0 \in \{ 1, 7, 13, 19 \}$  and $ \delta \in \{ 0, 1 \}$, we can always find an $\ell$ with (\ref{cookie}) and $\ell \leq 241$. From (\ref{final}), we thus have
$$
p \leq 242 + 2 \sqrt{241} < 5^{10}.
$$
For $q \geq 7$, we apply Proposition \ref{chumpy} to (\ref{final}) to conclude that
$$
p < e^q+1+2 \sqrt{e^q} = (e^{q/2}+1)^2 < q^{2q}.
$$
This completes the proof of Theorem \ref{main1}.

\section{Sets of primes and trivial solutions} \label{freebie}

\subsection{Intersections of the $S_i$} \label{sets} 
We would like to make a few remarks on the sets $S_i$.
Firstly, we note that some of the $S_i$ overlap substantially. Obviously, primes of the form $(3^b+1)/4$ belong to both $S_6$ and $S_7$, while many primes in $S_1$ are also in $S_4$ (taking $d=1$). Additionally,  every prime $q \in S_5$ of the shape $q=3 d^2+2^a$ with $a=2$ or $a \geq 8$, and $d=\pm 3^k$ for $k$ an integer, is necessarily also in $S_2$.

For many  other $i \neq j$, the intersection $S_i \cap S_j$ is rather small. For future use, it will be helpful for us to record an explicit statement along these lines.
\begin{prop} \label{intersect}
We have
$$
\begin{array}{c}
S_1 \cap S_2 = \left\{ 5, 7, 11, 13, 23, 31, 37, 73 \right\}, \; S_1 \cap S_3 =  \left\{ 11 \right\},  \; S_1 \cap S_5 = \{ 7, 31 \}, \\
S_1 \cap S_7 = \{ 7, 37, 127 \},  \; S_1 \cap S_8 =  \{ 13 \}, \; S_2 \cap S_3 =\left\{ 11 \right\}, \;  S_3 \cap S_4 = \emptyset, \\
S_3 \cap S_5 = \{ 43 \}, \; S_3 \cap S_7 =  S_3 \cap S_8 = S_4 \cap S_5 =S_5 \cap S_8 = S_7 \cap S_8 =  \emptyset. \\
\end{array}
$$
\end{prop}

\begin{proof}
The desired conclusions for $S_1 \cap S_2$, $S_1 \cap S_3$ and $S_2 \cap S_3$ all follow from combining Theorems 1, 2 and 3 of Tijdeman and Wang \cite{TiWa} with Theorems 3, 4 and 5 of Wang \cite{Wang}. Further, the fact that
$$
S_3 \cap S_4 = S_3 \cap S_8 = S_4 \cap S_5 =S_5 \cap S_8 =  \emptyset
$$
is immediate from considering the corresponding equations modulo $8$.

If $q \in S_1 \cap S_5$, then there exist integers $a, b, \delta, d$ and $a_5$ with $a \in \{ 2, 3 \}$ or $a \geq 5$, $b \geq 0$, $\delta \in \{ 0, 1 \}$, $d \geq 1$ and $a_5 \in \{ 2, 4 \}$ or $a_5 \geq 8$ even, such that $2^a 3^b + (-1)^\delta = 3 d^2 + 2^{a_5}$. Modulo $4$, we have that $\delta=1$ and so, modulo $3$, $b=0$. It follows, modulo $8$, that $a_5=2$, so that $2^a = 3d^2 + 5$. We have
\begin{lemma} \label{frosty}
If $x$ and $y$ are positive integers such that $2^x=3y^2+5$, then 
$$
(x,y) \in \left\{ (3,1), \, (5,3), \, (9,13) \right\}. 
$$
\end{lemma}
\begin{proof} (of Lemma \ref{frosty})
Writing $x = 3x_1 + x_0$ for $x_0 \in \{ 0, 1, 2 \}$, we have that a solution to the equation $2^x=3y^2+5$ necessarily corresponds to an integer point on the (Mordell) elliptic curve $Y^2 = X^3 - 2^{2x_0} \cdot 3^3 \cdot 5$ (with $Y=2^{x_0} \cdot 3^2 \cdot y$ and $X= 3 \cdot 2^{x_0+x_1}$).  We can find the  integer points for each of these curves at
\url{http://www.math.ubc.ca/~bennett/BeGa-data.html} (see \cite{BeGh} for more details), whereby the stated conclusion obtains. 
\end{proof}
From this lemma, we therefore have $S_1 \cap S_5 = \{ 7, 31 \}$, as desired. If instead $q \in S_1 \cap S_7$, then we have integers $a, b, \delta$ and $d$, with $a \in \{ 2, 3 \}$ or $a \geq 5$, $b \geq 0$, $\delta \in \{ 0, 1 \}$, $d \geq 1$ and $2^a 3^b + (-1)^\delta = \frac{3 d^2 +1}{4}$. If $b=0$, then, modulo $3$, $\delta=1$, whence $2^{a+2}=3d^2+5$ and so, from Lemma \ref{frosty}, $a \in \{ 1, 3, 7 \}$, corresponding to $q=7$ and $q=127$. If $b=1$, then, again modulo $3$, $\delta=0$ and hence $d^2=2^{a+2} 3^{b-1} + 1$. An elementary factoring argument implies that $(a,b,d) = (1,0,3), (1,2,5)$ or $(2,2,7)$, corresponding to $q=37$.

Suppose next that we have $q \in S_1 \cap S_8$, so that there exist integers $a, b, \delta, v$ with $a \in \{ 2, 3 \}$ or $a \geq 5$, $b \geq 0$, $\delta \in \{ 0, 1 \}$ and $2^a 3^b + (-1)^\delta = \frac{3 v^2 -1}{2}$. Modulo $4$, $\delta =0$ and hence $2^{a+1} 3^{b-1} + 1 = v^2$; again via factoring $v^2-1$ we find, after a little work, that $|v|=3$ and $q=13$. If $q \in S_3 \cap S_5$, there are integers $a, d $ and $a_5$ with $a \geq 5$ odd,  $d \geq 1$, $a_5 \in \{ 2, 4 \}$ or $a_5 \geq 8$ even, and $\frac{2^a+1}{3} = 3d^2 + 2^{a_5}$. Modulo $8$, we have that $a_5 \geq 4$ and so 
\begin{equation} \label{eq35}
9 d^2 = 2^a - 3 \cdot 2^{a_5} + 1. 
\end{equation}
We thus have $a \geq a_5 + 3$ and there must exist integers $\delta \in \{ 0, 1 \}$  and positive $d_1 \equiv \pm 1 \mod{6}$ such that $3 d = 2^{a_5-1} d_1 + (-1)^\delta$, whence
$$
2^{a_5-2} d_1^2 + (-1)^\delta d_1 = 2^{a-a_5} - 3.
$$
If $d_1=1$, then $\delta=0$ and we have $2^{a_5-4} + 1 = 2^{a-a_5-2}$, so that $a_5=4$, $a=7$, $d=3$ and $q=43$. If $d_1 > 1$, then $d_1 \geq 5$ and so $2^{a_5-2} 5^2 -5  \leq  2^{a-a_5} - 3$. It follows that $a \geq 2 a_5+2$ and hence
$2^{a_5-2} \mid (-1)^\delta d_1 + 3$, say $(-1)^\delta d_1 + 3 = 2^{a_5-2} d_2$, for $d_2 \in \mathbb{Z}$. We thus have
$$
2^{2a_5-4} d_2^2 - 3 \cdot 2^{a_5-1} d_2 + 9  + d_2 = 2^{a-2a_5+2}.
$$
Since $a_5 \geq 4$ and $a \geq 2 a_5+2$, we have that $9+d_2 \equiv 0 \mod{8}$. If $d_2=-1$, $2^{2a_5-7} + 3 \cdot 2^{a_5-4} + 1 = 2^{a-2a_5-1}$, contradicting $a_5=4$ or $a_5 \geq 8$. We thus have $d_2=7$ or $|d_2| \geq 9$, so that
$$
2^{a-2a_5+2} \geq 2^{2a_5-4} \cdot 7^2  - 3 \cdot 2^{a_5-1} \cdot 7 + 16 > 2^{2a_5 +1},
$$
and so $a \geq 4 a_5+4$. Applying Corollary 1.7 of \cite{BaBe}, since $a \geq 4 a_5+4 \geq 24$, we have from (\ref{eq35}) that
$$
3 \cdot 2^{a_5} -1 = \left| 9d^2-2^a \right| > 2^{0.26 a} \geq 2^{1.04 a_5+1.04},
$$
whereby $a_5 \leq 13$. A short check confirms that $S_3 \cap S_5 = \{ 43 \}$, as stated.

For $q \in S_3 \cap S_7$, $q = \frac{2^a+1}{3} = \frac{3d^2+1}{4}$ for $a \geq 5$ and $d$ odd integers, so that $2^{a+2}+1 =9d^2$ and yet another elementary argument implies that $a=1$, a contradiction. Let us therefore suppose, finally, that $q \in S_7 \cap S_8$. We thus have
$$
4 q = 3d^2+1 = 6 v^2 - 2 = 2 u^2 + 2,
$$
for integers $d, u$ and $v$, and so
$$
(3 d v)^2 = (2u^2+1)(u^2+2) = 2 u^4 + 5 u^2 + 2.
$$
From Magma's {\it IntegralQuarticPoints} routine, we find that the only integer solution to the latter equation is with $|dv|=|u|=1$. This completes the proof of Proposition \ref{intersect}.
\end{proof}

It should also be noted that representations within a given set $S_i$ are sometime unique, but not always. In particular, it is straightforward to show that a given prime $q \in S_1$ has a single representation of the form
$q=2^a 3^b \pm 1$, with $b \geq 0$ and $a \in \{ 2, 3 \} \mbox{ or } a \geq 5$, while a similar conclusion is immediate for primes $q \in S_i$ for $i \in \{ 3, 7, 8 \}$. The situation in $S_2$ is slightly more complicated; combining work of Pillai \cite{Pi} with Stroeker and Tijdeman \cite{StTi}, the only primes with multiple representations of the form $q = |2^a \pm 3^b|$, with $b \geq 1$ and  $a \in \{ 2, 3 \} \mbox{ or } a \geq 5$, are $q \in \{ 5, 13, 17, 23, 73 \}$, corresponding to the identities
$$
5 = 2^3-3 = 2^5-3^3 = 3^2-2^2,
$$
$$
13 = 2^8-3^5 = 2^2+3^2,
$$
$$
17 = 3^4-2^6 = 2^3+3^2,
$$
$$
23 = 2^5-3^2 = 3^3-2^2,
$$
and
$$
73=3^4-2^3 = 2^6+3^2.
$$

\subsection{Limitations due to trivial solutions} \label{trivial}
Notice that we have the identity
$$
\left( \frac{d+1}{2} \right)^3 +  \left( \frac{1-d}{2} \right)^3 = \frac{3d^2+1}{4}
$$
and hence, for all exponents $n$, a coprime integer solution with $C=1$ to the equation 
$$
A^3+B^3 = q^{\alpha} \cdot C^n, \; \; \mbox{ whenever } \; \; q^{\alpha} =\frac{3d^2 +1}{4}.
$$
We expect $q^\alpha$ to be of this shape infinitely often for $\alpha=1$ and $\alpha=2$ (these are precisely the primes in $S_7$ and $S_8$, respectively), though both of these results are a long way from provable with current technology. 

We will term  a solution to \eqref{ABC} with $C=1$ {\it trivial}, whereby, for primes $q$ as above, there exists a trivial solution for all prime exponents $p$. 
In particular, this means that one of the newforms $f\in S_2^+(N_{F}/\mathcal{R})$ (see Section~\ref{Frey}) will 
correspond (via modularity) to the Frey curve $F$ evaluated at the trivial solution.
This is a major obstruction to the modular method;
the techniques of this paper are unlikely to provide further information 
about \eqref{ABC} with $\alpha=1$ for $q \in S_7$ and $\alpha=2$ for $q \in S_8$.

\smallskip

A similar relation is the identity
$$
(d+4)^3+ (4-d)^3 = 8 \left( 3d^2+16 \right).
$$
While this does not actually give trivial solutions to (\ref{ABC}) in case $\alpha =1$ and $q = 3d^2+16$ (a subset of the primes in $S_5$), it does appear to provide an obstruction to solving (\ref{gump}) for such primes, leading to Frey-Hellegouarch curves that play the role of the curve $72A1$ for equation (\ref{start}).

\section{ Applying the symplectic criteria} \label{apply}

Let $E$ and $F$ be elliptic curves over $\Q$ and suppose there exists an isomorphism 
$\phi : F[p] \to E[p]$ of $\Gal(\Qbar / \Q)$-modules. Here, $F[p]$ and $E[p]$ are the $p$-torsion modules attached to $F$ and $E$, respectively. 
Write $e_{E,p}$ and $e_{F,p}$ for the Weil pairings on $E[p]$ and $F[p]$, respectively. Then there exists an element $r (\phi) \in \F_p^{\times}$ such that
$$
e_{E,p} \left( \phi (P), \phi(Q) \right) = e_{F,p} \left( P, Q \right)^{r (\phi)} \; \mbox{ for all } \; P, Q \in F[p].
$$
If $r(\phi)$ is a square in $\F_p^\times$, we call the isomorphism $\phi$ {\it symplectic}; if $r(\phi)$ is a non-square, we call it {\it anti-symplectic}. We say that $E[p]$ and $F[p]$ are {\it symplectically (anti-symplectically) isomorphic} if there exists a symplectic (anti-symplectic) isomorphism $\phi$ between them.
It is possible that $E[p]$ and $F[p]$ are both symplectically and anti-symplectically isomorphic, but this situation will not occur in the applications of these techniques in this paper. 

\subsection{The symplectic argument}
To treat equation \eqref{ABC} for certain primes $q \in S_0$ and exponents $p \geq q^{2q}$ we 
need to use a number of local symplectic criteria to describe the symplectic type of the isomorphisms between the $p$-torsion modules $E[p]$ and $F[p]$, where $F$ is our Frey-Hellegouarch curve and $E$ is one of the curves in Corollary~\ref{wonder} (see Section~\ref{Frey} and Theorem~\ref{main1}). 
The idea is to use local information at different primes $\ell$ to 
obtain congruence conditions on the exponent $p$ for which $E[p]$ and $F[p]$ are symplectically and 
anti-symplectically isomorphic. Then, our desired contradictions will arise each time we are able to prove that these constraints are incompatible. This is, in essence, what is sometimes called the {\it symplectic argument}.
One advantage we have here, working with equation (\ref{ABC}) as opposed to equation (\ref{start}), is that we will be able to apply the (local) criteria at the primes $\ell \in \{ 2, 3, q \}$ rather than just $\ell \in \{ 2, 3 \}$. 

\subsection{Notation} Let $\ell$ be a prime and, for a nonzero integer $t$, define $\nu_{\ell} (t)$ to be the largest nonnegative integer such that $\ell^{\nu_{\ell}(t)}$ divides $t$.
Let $E/\Q_\ell$ be an elliptic curve and write $c_4 (E)$, $c_6 (E)$ and $\Delta (E)$
for the usual invariants attached to a minimal model of $E$.
Further, define the quantities $c_4 (E)_{\ell}$, $c_6 (E)_{\ell}$ 
and $\Delta (E)_{\ell}$ by
$$
c_4 (E)  = \ell^{\nu_{\ell}(c_4 (E) )} c_4 (E)_{\ell}, \; \;  c_6 (E)  = \ell^{\nu_{\ell}(c_6 (E) )} c_6 (E)_{\ell} \; \mbox{ and } \;  \Delta (E)  =\ell^{\nu_{\ell}(\Delta (E) )} \Delta (E)_{\ell}.
$$
Let $\Q_{\ell}^{\text{un}}$ to be the maximal unramified extension of $\Q_{\ell}$. For an elliptic 
curve $E/\Q$ with potentially good reduction at $\ell$ we write $e(E,\ell)$ to denote the order of $\Gal (\Q_{\ell}^{\text{un}}(E[p])/\Q_{\ell}^{\text{un}})$ for $p \geq 3$ different from $\ell$. It is well known that $e(E,\ell)$ is independent of $p$.

\subsection{The curves}
Except for the few isogeny classes given in Corollary~\ref{wonder} 
by their Cremona label, from Theorem \ref{main1}, we are 
primarily interested in applying symplectic criteria to our Frey-Hellegouarch curve and the isogeny classes of curves
$$
\begin{array}{c}
18q.1.a, 18q.2.a, 18q.3.a, 18q.4.a \; (\mbox{with } \delta_1=\delta_2=0, a \mbox{ even}, b \mbox{ odd}), 18.q.5.a \mbox{ and }  \\
18q.5.b \; (\mbox{with, in both cases, } \delta_1=\delta_2=0 \mbox{ and } a \mbox{ even}), 36q.1.a, 36q.1.b, 36q.2.a,   \\
36q.2.b, 36q.3.a,  72q.1.a, 72q.2.a, 72q.3.a, 72q.4.a, 72q.4.b, 72.q.5.a  \mbox{ and } 72q.5.b \\
(\mbox{with, in both cases, } \delta=0 \mbox{ and } a =4), 72q.6.a, 72q.6.b, 72q.7.a, 72q.8.a  \\
 \mbox{ and }  72q.9.a \; (\mbox{with } \delta_1=\delta_2=0, a =4, \, b \mbox{ odd}).\\
\end{array}
$$
The relevant arithmetic data $c_4(E)$, $c_6(E)$ and $\Delta (E)$ is available in our Appendix and in the statements of Theorems \ref{theorem18q}, \ref{theorem36q} and \ref{theorem72q}. 
In the remainder of this section we will apply the criteria to the curves listed above 
to obtain congruence conditions on $p$. Then, in Section~\ref{applications}, we complete the symplectic argument by deriving contradictions from these conditions, allowing us to finish 
the proofs of our main Diophantine statements. We start by proving the following proposition which
holds for all our choices of $E$, independent of whether $E$ has conductor $N_E=18q$, $36q$ or $72q$.

\begin{prop} \label{multiplicativeAtq} 
Let $(A,B,C)$ be a non-trivial primitive solution to \eqref{ABC} so that
there is a $\Gal(\Qbar/ \Q)$-modules isomorphism $\phi \; : \; F[p] \to E[p]$, where $F$ 
is the Frey-Hellegouarch curve and $E$ is any elliptic curve in one of the isogeny classes above. 
Then
\[\phi \; \text{is symplectic} \quad \iff  \quad \alpha
\text{ is a square mod } p.\]
\end{prop}
\begin{proof} We have $\nu_q(N_F) = \nu_q(N_E) = 1$, so $q$ is a prime of multiplicative
reduction of both curves. We can always choose $E$ such that $\nu_q(\Delta(E)) = 2$;
moreover, we have $p \nmid \alpha$ and 
$\nu_q(\Delta(F)) = 2\alpha + 2p\nu_q(C)$.
The conclusion now follows from a direct application of
\cite[Proposition~2]{KO} with the prime $q$.
\end{proof}

\subsection{Curves of conductor $18q$}

We  summarize the necessary information about the invariants of the relevant elliptic curves.
\begin{small}
$$
\begin{array}{|c|ccc|ccc|}  \hline
\mbox{Curve} & \nu_2 (c_4) & \nu_2 (c_6) & \nu_2 (\Delta) & \nu_3 (c_4) & \nu_3 (c_6) & \nu_3 (\Delta) \\ \hline
F^{(0)}_{A,B} & 0 & 0 &  2p\nu_2(C) - 8 & 2 + \nu_3 (AB)  & 3 + \nu_3 (A^3-B^3) & 2p\nu_3(C) + 3  \\ \hline
18q.1.a1 \, (b=0) & 0 & 0 & 2a-8 & 3 & \geq 7 & 6 \\
18q.1.a1 \, (b \geq 1) & 0 & 0 & 2a-8 & 2 & 3 & 2b+6 \\
18q.2.a1 & 0 & 0 & 2a-8 & 2 & 3 & 8 \\
18q.3.a1 & 0 & 0 & 2a-8 & 2 & 3 & 2b+6 \\
18q.4.a2 & 0 & 0 & a-6 & 2 & 3 & b+6 \\
18q.5.a2 & 0 & 0 & a-6 & 2 & 3 + \nu_3 (d)  & 3 \\ 
18q.5.b2 & 0 & 0 & a-6 & 4 & 6  + \nu_3(d) & 9 \\ \hline
\end{array}
$$
\end{small}

Suppose $(A,B,C)$ is a non-trivial primitive solution to \eqref{ABC}
and the Frey-Hellegouarch curve $F$ satisfies isomorphism \eqref{iso1}
where $f$ is the newform corresponding to one of the isogeny classes 
$$
18q.1.a, 18q.2.a, 18q.3.a, 18q.4.a, 18q.5.a \mbox{ or }  18q.5.b.
$$
In particular, $F=F^{(0)}_{A,B}$, $C$ is even and $B \equiv -1 \mod{4}$. 
Moreover, there is a $\Gal(\Qbar/ \Q)$-modules isomorphism
$\phi  :  F[p] \to E[p]$,
where $E$ is one of the elliptic curves  
$$
18q.1.a1, 18q.2.a1, 18q.3.a1, 18q.4.a2, 18q.5.a2  \mbox{ or }  18q.5.b2.
$$ 

\vskip2ex
\subsubsection{Applying the criteria at $\ell=2$} \label{multiplicative}
Since $\nu_2(N_E) = 1$ the prime $\ell=2$ is of multiplicative reduction for $E$. From \cite[Proposition~2]{KO} and the valuations given in the preceding table, it follows 
that either $p \nmid a-4$ and
$$
\phi \; \text{is symplectic} \quad \iff  \quad 4-a
\text{ is a square mod } p,
$$
in  case $E = 18q.1.a1, 18q.2.a1$ or $18q.3.a1$,
or that $p \nmid a-6$ and 
$$
\phi \; \text{is symplectic} \quad  \iff  \quad 12-2a
\text{ is a square mod } p,
$$
in the other cases.

\vskip2ex
\subsubsection{Applying the criteria at $\ell=3$} \label{inertia4At3}
We first consider $E$ one of $18q.1.a1$ with $b=0$,
$18q.5.a2$ or $18q.5.b2$. We have that the corresponding $j$-invariant satisfies $\nu_3(j_E) > 0$,  and hence $E$ has potentially good reduction at $3$. 
Indeed, for $E = 18q.1.a1$ (with $b=0$), we have $\nu_3 (\Delta(E)) =6$ and 
$\nu_3(c_6(E)) \geq 7$ so that,  from \cite[p. 356]{Kraus1990}, we conclude that $e(E,3)=2$.

For $E=18q.5.a2$ and $E=18q.5.b2$, we have $\nu_3(\Delta(E)) \in \{ 3, 9 \}$ and
the results of  \cite[p. 356]{Kraus1990} imply that $e(E,3) \in \{ 4, 12 \}$;
furthermore, since $\nu_3(N_E) =  2$ we are in a case 
of tame reduction, whence $e(E,3) = 4$. On the other hand, for our Frey-Hellegouarch 
curve $F$ to have potentially good reduction at $3$, we require that $3 \nu_3 (c_4(F)) \geq \nu_3 (\Delta (F))$, or, equivalently, $\nu_3 (C)=0$. In this situation, $\nu_3(\Delta(F)) = 3$ and arguing exactly as for the previous curves we also conclude that $e(F,3) = 4$. This contradicts $e(E,3)=2$, and hence $E \neq 18q.1.a1$ (with $b=0$). 
We will now apply \cite[Theorem~1]{FrKr2} with $F$ and 
$E = 18q.5.a2$ or $18q.5.b2$ (with, in both cases, $\delta_1 = \delta_2 = 0$ 
and $a$ even). Let $r$ and $t$ be the quantities defined in the statement of that theorem.
We have, since $3 \nmid C$,
$$
 \nu_3(\Delta(F)) = \nu_3(\Delta(18q.5.a2)) = 3 \; \mbox{ and } \;
 \nu_3(\Delta(18q.5.b2)) = 9,
$$
whereby $r=0$ if $E = 18q.5.a2$ and $r=1$ if $E = 18q.5.b2$.
Moreover, since $3 \nmid C$ and $a$ is even, we may check that
$\Delta (F)_3 \equiv \Delta (E)_3 \equiv 2 \pmod{3}$, i.e.
$t=0$ for both $E$. Finally, applying \cite[Theorem~1]{FrKr2},
we conclude that $\phi$ is symplectic when $E = 18q.5.a2$ and,
if $E = 18q.5.b2$, then $\phi$ is symplectic if and only if $(3/p) = 1$.

\medskip

We now consider the remaining curves $E$ of conductor $18q$ under consideration. We have, in all cases,
$$
\nu_3(c_4(E)) = 2, \quad \nu_3(c_6(E)) = 3, \quad \nu_3(\Delta (E)) \geq 7 \; \mbox{ and } \; \nu_3(j_E)  < 0,
$$
and hence $E$ has potentially multiplicative reduction at $3$; after a quadratic twist (with corresponding elliptic curve denoted $E_t$) the reduction becomes multiplicative and we have 
$$
\nu_3(\Delta (E_t)) = 
\left\{ 
\begin{array}{l} 
2 \mbox{ if } E = 18q.2.a1, \\
b  \mbox{ if } E = 18q.4.a2 \ (\mbox{with } b \geq 1), \\
2b  \mbox{ if } E = 18q.1.a1 \, (\mbox{and } b \geq 1) \mbox{ or } E=18q.3.a1.  \\
\end{array}
\right.
$$
Furthermore,
$3$ must divide $C$ (since otherwise $F$ would have potentially good reduction) and twisting the Frey curve $F$ by the same element (to obtain $F_t$), we find that
$\nu_3(\Delta (F_t)) = -3 + 2p\nu_3(C)$.

If $E = 18q.1.a1$ with $b \geq 1$ or $E=18q.3.a1$, it follows from \cite[Proposition~2]{KO} applied to $E_t$ and $F_t$ that $p \nmid b$ and
$$
\phi \; \text{is symplectic} \quad \iff  \quad -6b
\text{ is a square mod } p.
$$
Similarly, if $E = 18q.4.a2$ then $p \nmid b$ and 
$$
\phi \; \text{is symplectic} \quad \iff  \quad -3b
\text{ is a square mod } p.
$$
If $E = 18q.2.a1$, then
$$
\phi \; \text{is symplectic} \quad \iff  \quad -6
\text{ is a square mod } p.
$$

\subsubsection{Conclusions for level $18q$} \label{fool18}
From the calculations above and Proposition~\ref{multiplicativeAtq} we 
can extract the following relations. 
If $E = 18q.1.a1$ or $18.3.a1$ then $b \geq 1$ and 
\begin{equation} \label{sample}
\left( \frac{4-a}{p} \right) = \left( \frac{\alpha}{p} \right)  = \left( \frac{-6b}{p} \right),
\end{equation}
while if $E=18.4.a2$ or $E = 18q.2.a1$, then, respectively, 
$$
\left( \frac{12-2a}{p} \right) = \left( \frac{\alpha}{p} \right)  = \left( \frac{-3b}{p} \right) \quad \text{or} \quad 
\left( \frac{4-a}{p} \right) = \left( \frac{\alpha}{p} \right)  = \left( \frac{-6}{p} \right).
$$
If $E=18.5.a2$, we have that 
$$
 \left( \frac{12-2a}{p} \right) = \left( \frac{\alpha}{p} \right)  = 1.
$$

\subsection{Curves of conductor $36q$}

We next proceed with the case of elliptic curves of conductor $36q$. We encounter the following invariants. 
$$
\begin{array}{|c|ccc|ccc|}  \hline
\mbox{Curve} & \nu_2 (c_4) & \nu_2 (c_6) & \nu_2 (\Delta) & \nu_3 (c_4) & \nu_3 (c_6) & \nu_3 (\Delta) \\ \hline
F^{(1)}_{A,B} & 4 + \nu_2 (A) & 5 &  4 & 2 + \nu_3 (AB)  & 3 + \nu_3 (A^3-B^3) & 2p\nu_3(C) + 3  \\ \hline
36q.1.a2 \, (3 \mid s) & 4 & 6 & 8 & 2 & 4+  \nu_3 (v) & 3 \\
36q.1.a2 \, (3 \not\;\mid s) & 4 & 6 & 8 & 2 & 3 & 3 \\ 
36q.1.b2 \, (3 \mid s) & 4 & 6 & 8 & 4 & 7+  \nu_3 (v) & 9 \\ 
36q.1.b2 \, (3 \not\;\mid s) & 4 & 6 & 8 & 4 & 6 & 9 \\ 
36q.2.a1 \, (3 \mid d) & \geq 6 & 5 & 4 & 2 & 4 + \nu_3 (d) & 3 \\
36q.2.a1 \, ( 3 \not\;\mid d) & \geq 6 & 5 & 4 & \geq 3  & 3  & 3 \\
36q.2.b1 \, (3 \mid d) & \geq 6 & 5 & 4 & 4 & 7 + \nu_3 (d) & 9 \\
36q.2.b1 \, ( 3 \not\;\mid d) & \geq 6 & 5 & 4 & \geq 5  & 6  & 9 \\
36q.3.a1 & \geq 6 & 5 & 4 & 2 & 3 & b+6 \\
\hline
\end{array}
$$

Suppose $(A,B,C)$ is a non-trivial primitive solution to \eqref{ABC}
and the Frey-Hellegouarch curve $F$ satisfies isomorphism \eqref{iso1}
where $f$ is the newform corresponding to one of the isogeny classes
$$
36q.1.a, 36q.1.b, 36q.2.a, 36q.2.b \; \mbox{ or } \; 36q.3.a.
$$
In particular, $F=F^{(1)}_{A,B}$, $C$ is odd and $B \equiv 1 \mod{4}$. 
Moreover, there is a $\Gal(\Qbar/ \Q)$-modules isomorphism
$\phi : F[p] \to E[p]$,
where $E$ is one of the elliptic curves  
$$
36q.1.a2, \; 36q.1.b2, \; 36q.2.a1, \; 36q.2.b1 \; \mbox{ or } \; 36q.3.a1.
$$ 

\vskip2ex
\subsubsection{Applying the criteria at $\ell=2$}  \label{inertia2}
The table shows that $\nu_2(j(E)) > 0$ for all $E$, so that
the curves have potentially good reduction. Since $\nu_2(N_E) = 2$
the reduction is tame and hence $e(E,2) = 3$ for all $E$.

We will now apply \cite[Theorem~1]{FrKr1} at $\ell = 2$ with $F$ and
$E=36q.1.a2$ or $36q.1.b2$. Let $t$ and $r$ be as in that theorem. 
Since $\nu_2(\Delta(F)) = 4$ and $\nu_2(\Delta(E)) = 8$, we have $r=1$ for both $E$.
Now, to determine the value of $t$, we must first appeal to \cite[Theorem~3]{FrKr1}.
Indeed, the curve $36q.1.a2$ has
$$
c_4(E)_2 = 3 \left( 16 \left(  \frac{3v^2-1}{2} \right)^2 - 1 \right)
 \; \mbox{ and } \; c_6 (E)_2 =  - 3^2 u v \left( 32 \left(  \frac{3v^2-1}{2} \right)^2 + 1 \right),
$$
while for $36q.1.b2$, 
$$
c_4(E)_2  = 3^3 \left( 16 \left(  \frac{3v^2-1}{2} \right)^2 - 1 \right)
 \; \mbox{ and } \; c_6 (E)_2 =  3^5 u v \left( 32 \left(  \frac{3v^2-1}{2} \right)^2 + 1 \right).
$$
We thus have, respectively,
$$
c_4(E)_2  \equiv 13 \mod{32} \; \mbox{ and } \; c_6 (E)_2 \equiv 7 u v \mod{16},
$$
and
$$
c_4(E)_2  \equiv 21 \mod{32} \; \mbox{ and } \; c_6 (E)_2 \equiv 3 u v \mod{16}.
$$
Since $u^2-3v^2=-2$ with $u \equiv v \equiv 1 \mod{4}$, the $u$ and $v$ are terms in binary recurrence sequences and we may readily prove by induction that $uv \equiv 1 \mod{16}$, whereby it follows from \cite[Theorem~3]{FrKr1} that
the curve $36q.1.a2$ has a $3$-torsion point over $\mathbb{Q}_2$, while $36q.1.b2$ does not. 
For $F=F^{(1)}_{A,B}$, since $\nu_2(A) \geq 2$, we have 
\[ \nu_2(c_4(F)) \geq 6 \quad \text{ and } \quad c_6 (F)_2 = 3^3(A^3-B^3) \equiv -3 B \mod{8},\]
whereby, from  part (B2) of \cite[Theorem~3]{FrKr1}, $F$ has a $3$-torsion point over $\mathbb{Q}_2$ 
precisely when we have $B \equiv 1 \mod{8}$.
We thus conclude that, if $B \equiv 1 \mod{8}$ and $E=36q.1.a2$
or $B \equiv 5 \mod{8}$ and $E=36q.1.b2$, then $t=0$ and
$$
\phi \; \text{is symplectic} \quad \iff  \quad 2 \text{ is a square mod } p.
$$
If $B \equiv 5 \mod{8}$ and $E=36q.1.a2$, or $B \equiv 1 \mod{8}$ and $E=36q.1.b2$, then $t=r=1$ and so
$$
\phi \; \text{is symplectic} \quad \iff  \left( \frac{2}{p} \right) = \left( \frac{3}{p} \right).
$$
Next, suppose that $E$ is one of $36q.2.a1, 36q.2.b1$ or $36q.3.a1$, so that we always have $r=0$.
Then
$$
c_6 (E)_2 = -3^3 d \frac{(d^2+3)}{4} \equiv 
\begin{cases}
5 \mod{8} \mbox{ if } d \equiv 1 \mod{16}, \\
1 \mod{8} \mbox{ if } d \equiv 9 \mod{16}, \\
\end{cases}
$$
$$
c_6 (E)_2 = 3^6 d \frac{(d^2+3)}{4} \equiv 
\begin{cases}
5 \mod{8} \mbox{ if } d \equiv 9 \mod{16}, \\
1 \mod{8} \mbox{ if } d \equiv 1 \mod{16}, \\
\end{cases}
$$
and
$$
c_6 (E)_2  = -3^3 d \frac{(d^2+3^{b+2})}{4} \equiv 
\begin{cases}
5 \mod{8} \mbox{ if } q \equiv d+2 \mod{8}, \\
1 \mod{8} \mbox{ if } q \equiv d-2 \mod{8}, \\
\end{cases}
$$
respectively. If any of
$$
\begin{array}{l}
B \equiv 1 \mod{8}, \; \; E=36q.2.a1 \;  \mbox{ and }  \; d \equiv 1 \mod{16}, \mbox{ or } \\
B \equiv 5 \mod{8}, \; \; E=36q.2.a1 \;  \mbox{ and }  \; d \equiv 9 \mod{16}, \mbox{ or } \\
B \equiv 1 \mod{8}, \; \; E=36q.2.b1 \;  \mbox{ and }  \; d \equiv 9 \mod{16}, \mbox{ or } \\
B \equiv 5 \mod{8}, \; \; E=36q.2.b1 \;  \mbox{ and }  \; d \equiv 1 \mod{16}, \mbox{ or } \\
B \equiv 1 \mod{8}, \; \; E=36q.3.a1 \;  \mbox{ and }  \; q \equiv d+2 \mod{8}, \mbox{ or } \\
B \equiv 5 \mod{8}, \; \; E=36q.3.a1 \;  \mbox{ and }  \; q \equiv d-2 \mod{8}, \\
\end{array}
$$
we therefore have that $\phi$ is always symplectic. If, however,
$$
\begin{array}{l}
B \equiv 1 \mod{8}, \; \; E=36q.2.a1 \;  \mbox{ and }  \; d \equiv 9 \mod{16}, \mbox{ or } \\
B \equiv 5 \mod{8}, \; \; E=36q.2.a1 \;  \mbox{ and }  \; d \equiv 1 \mod{16}, \mbox{ or } \\
B \equiv 1 \mod{8}, \; \; E=36q.2.b1 \;  \mbox{ and }  \; d \equiv 1 \mod{16}, \mbox{ or } \\
B \equiv 5 \mod{8}, \; \; E=36q.2.b1 \;  \mbox{ and }  \; d \equiv 9 \mod{16}, \mbox{ or } \\
B \equiv 1 \mod{8}, \; \; E=36q.3.a1 \;  \mbox{ and }  \; q \equiv d-2 \mod{8}, \mbox{ or } \\
B \equiv 5 \mod{8}, \; \; E=36q.3.a1 \;  \mbox{ and }  \; q \equiv d+2 \mod{8}, \\
\end{array}
$$
then 
$$
\phi \; \text{is symplectic} \quad \iff  \quad 3 \text{ is a square mod } p.
$$

\vskip2ex
\subsubsection{Applying the criteria at $\ell=3$} 
For $E = 36q.3.a1$, we have $\nu_3 (j(E)) < 0$ and so $E$ has potentially multiplicative reduction at $3$. After a suitable quadratic twist (denoted $E_t$) the reduction 
becomes multiplicative and $\nu_3(\Delta (E_t)) = b$. 
Therefore, the twisted Frey curve $F_t$ must also have 
multiplicative reduction at $3$ (since $p \geq 5$)
and it satisfies $\nu_3(\Delta (F_t)) = 2p \nu_3(C)-3$.
Since $p \nmid \nu_3(\Delta (F_t))$, it follows from \cite[Proposition~2]{KO}
that $p \nmid b$ and
$$
\phi \; \text{is symplectic} \quad \iff  \quad -3b \text{ is a square mod } p.
$$

For all other cases of $E$ we have $\nu_3(j(E)) \geq 0$ and $\nu_3(\Delta(E)) \ne 6$, 
whence $E$ has potentially good reduction which does not become good after
a quadratic twist. As before, since $\nu_3(N_E) = 2$ the reduction is tame, whereby $e(E,3)=4$.
A similar argument guarantees that $e(F,3) =4$ when $3 \nmid C$, in which case, $\nu_3(\Delta(F)) = 3$ 
and $\Delta(F)_3 \equiv 2 \pmod{3}$. To apply \cite[Theorem~1]{FrKr2} at $\ell=3$ with $F$ and
each of the curves $E=36q.1.a2, 36q.1.b2, 36q.2.a1$ or $36q.2.b1$, we first compute that
$(r,t)=(0,1), (1,1), (0,0)$  and $(1,0)$, respectively. We conclude that
if $E=36q.2.a1$ then $\phi$ is symplectic, while, if $E=36q.1.a2$, 
$$
\phi \; \text{is symplectic} \quad \iff  \quad 2 \text{ is a square mod } p.
$$
If $E=36q.1.b2$, then
$$
\phi \; \text{is symplectic} \quad \iff  \left( \frac{2}{p} \right) = \left( \frac{3}{p} \right)
$$
and if $E=36q.2.b1$, then
$$
\phi \; \text{is symplectic} \quad \iff  \quad 3 \text{ is a square mod } p.
$$

\subsubsection{Conclusions for level $36q$} \label{fool36}
From the calculations above and Proposition~\ref{multiplicativeAtq} 
we can extract the following relations. 
If $E=36q.1.a2$ and $B \equiv 1 \mod{8}$, we have
$$
\left( \frac{\alpha}{p} \right)  = \left( \frac{2}{p} \right) 
$$
while $E=36q.1.a2$ and $B \equiv 5 \mod{8}$ implies that either
$$
\left( \frac{\alpha}{p} \right)  = \left( \frac{2}{p} \right)  = \left( \frac{3}{p} \right)  = 1, \; \mbox{ or } \; \left( \frac{\alpha}{p} \right)  = \left( \frac{2}{p} \right)  = -1, \;  \left( \frac{3}{p} \right)  = 1.
$$
If $E=36q.1.b2$ and $B \equiv 1 \mod{8}$, we have either
$$
\left( \frac{\alpha}{p} \right)  = 1, \;  \left( \frac{2}{p} \right)  = \left( \frac{3}{p} \right), \; \mbox{ or } \; \left( \frac{\alpha}{p} \right)  = -1, \; \left( \frac{2}{p} \right)  \neq \left( \frac{3}{p} \right).
$$
If $E=36q.1.b2$ and $B \equiv 5 \mod{8}$, we either have 
$$
\left( \frac{\alpha}{p} \right)  = \left( \frac{2}{p} \right) = \left( \frac{3}{p} \right) =
1\; \mbox{ or } \;  \left( \frac{\alpha}{p} \right)  = \left( \frac{2}{p} \right) = -1, \; \left( \frac{3}{p} \right) = 1.
$$
If $E=36q.2.a1$ and either $B \equiv 1 \mod{8}$, $d \equiv 1 \mod{16}$, or $B \equiv 5 \mod{8}$, $d \equiv 9 \mod{16}$, we have 
$$
\left( \frac{\alpha}{p} \right)  = 1.
$$
If $E=36q.2.a1$ and either $B \equiv 1 \mod{8}$, $d \equiv 9 \mod{16}$, or $B \equiv 5 \mod{8}$, $d \equiv 1 \mod{16}$, we have 
$$
\left( \frac{\alpha}{p} \right)  = \left( \frac{3}{p} \right)  = 1.
$$
If $E=36q.2.b1$ and either $B \equiv 1 \mod{8}$, $d \equiv 9 \mod{16}$, or $B \equiv 5 \mod{8}$, $d \equiv 1 \mod{16}$, we have, again,
$$
\left( \frac{\alpha}{p} \right)  = \left( \frac{3}{p} \right)  = 1,
$$
while, if $E=36q.2.b1$ and either $B \equiv 1 \mod{8}$, $d \equiv 1 \mod{16}$, or $B \equiv 5 \mod{8}$, $d \equiv 9 \mod{16}$, we have
$$
\left( \frac{\alpha}{p} \right)  = \left( \frac{3}{p} \right)  = 1. 
$$
If $E=36q.3.a1$ and either $B \equiv 1 \mod{8}$, $q \equiv d+2 \mod{8}$, or $B \equiv 5 \mod{8}$, $q \equiv d-2 \mod{8}$, we have
$$
\left( \frac{\alpha}{p} \right)  = \left( \frac{-3b}{p} \right)  = 1,
$$
while, if $E=36q.3.a1$ and either $B \equiv 1 \mod{8}$, $q \equiv d-2 \mod{8}$, or $B \equiv 5 \mod{8}$, $q \equiv d+2 \mod{8}$, we have that 
$$
\left( \frac{\alpha}{p} \right)  = \left( \frac{3}{p} \right)  =  \left( \frac{-3b}{p} \right).
$$

\subsection{Curves of conductor $72q$}

We have the following data.

$$
\begin{array}{|c|ccc|ccc|}  \hline
\mbox{Curve} & \nu_2 (c_4) & \nu_2 (c_6) & \nu_2 (\Delta) & \nu_3 (c_4) & \nu_3 (c_6) & \nu_3 (\Delta) \\ \hline
F^{(1)} _{A,B} & 5 & 5 &  4 & 2 + \nu_3 (AB)  & 3 + \nu_3 (A^3-B^3) & 2p\nu_3(C) + 3  \\ \hline
72q.1.a1 & 4 & 6 & 8 & 2 & 3 & 2b+6 \\
72q.2.a1 & 4 & 6 & 8 \mbox{ or } 10 & 2 & 3 & 2b+6 \\
72q.3.a1 & 4 & 6 & 8 \mbox{ or } 10  & 2 & 3 & 2b+6 \\
72q.4.a2 \, (3 \mid d) & 4 & 6 & 8  & 2 & 4+ \nu_3 (d) & 3 \\
72q.4.a2 \, (3 \not\;\mid d) & 4 & 6 & 8  & \geq 3 & 3 & 3 \\
72q.4.b2 \, (3 \mid d) & 4 & 6 & 8  & 4 & 7+ \nu_3 (d) & 9 \\
72q.4.b2 \, (3 \not\;\mid d) & 4 & 6 & 8  & \geq 5 & 6 & 9 \\
72q.5.a2 \, (3 \mid d) & 4 & 6 & 10 & 2 & 4+ \nu_3 (d) & 3 \\
72q.5.a2 \, (3 \not\;\mid d) & 4 & 6 & 10 & \geq 3 & 3 & 3 \\
72q.5.b2 \, (3 \mid d) & 4 & 6 & 10 & 4 & 7+ \nu_3 (d) & 9 \\
72q.5.b2 \, (3 \not\;\mid d) & 4 & 6 &  10 & \geq 5 & 6 & 9 \\
72q.6.a1 \, (3 \mid d)  & 5 & 5 & 4 & 2 & 4 + \nu_3 (d)  & 3 \\ 
72q.6.a1 \, (3 \not\;\mid d)  & 5 & 5 & 4 & \geq 3 & 3  & 3 \\ 
72q.6.b1 \, (3 \mid d)  & 5 & 5 & 4 & 4 & 7  + \nu_3(d) & 9 \\
72q.6.b1 \, (3 \not\;\mid d)  & 5 & 5 & 4 & \geq 5 & 6 & 9 \\
72q.7.a1 & 5 & 5 & 4 & 2 & 3 & b+6 \\
72q.8.a2 & 4 & 6 & 8 & 2 & 3 & b+6 \\
72q.9.a2 & 4 & 6 & 10 & 2 & 3 & b+6 \\ \hline
\end{array}
$$

Suppose $(A,B,C)$ is a non-trivial primitive solution to \eqref{ABC}
and the Frey-Hellegouarch curve $F$ satisfies isomorphism \eqref{iso1}
where $f$ is the newform corresponding to one of the isogeny classes 
$$
\begin{array}{c}
72q.1.a, 72q.2.a, 72q.3.a, 72q.4.a, 72q.4.b, 72q.5.a, 72q.5.b, \\
72q.6.a, 72q.6.b, 72q.7.a, 72q.8.a  \mbox{ or }  72q.9.a. \\
\end{array}
$$
In particular, for this case we have $F=F_{A,B}^{(1)}$,
$$
 C \text{ is odd}, \qquad A \equiv 2 \mod{4} \; \; \mbox{ and } \; \;  B \equiv 1 \pmod{4},
$$
and there is a $\Gal(\Qbar/ \Q)$-module isomorphism
$$
  \phi \; : \; F[p] \to E[p],
$$
where $E$ is one of the elliptic curves labelled 
$$
\begin{array}{c}
72q.1.a1, 72q.2.a1, 72q.3.a1, 72q.4.a2, 72q.4.b2, 72q.5.a2, 72q.5.b2, \\
72q.6.a1,  72q.6.b1, 72q.7.a1, 72q.8.a2  \mbox{ or }  72q.9.a2. \\
\end{array}
$$ 

\vskip2ex
\subsubsection{Applying the criteria at $\ell=2$} 
Note that all the curves in the preceding table  
have potentially good reduction at $\ell=2$ since their $j$-invariants satisfy $\nu_2(j) \geq 0$. 
We see, from \cite[p. 358]{Kraus1990}, that the Frey curve $F$ satisfies $e(F, 2)=24$; 
the same is immediately seen to be true also for $E$ satisfying
$$
( \nu_2 (c_4(E)), \nu_2 (c_6(E)), \nu_2 (\Delta(E))) \in \{ (4, 6, 10), (5,5,4) \}. 
$$
For the curves $E$ in the table with 
$$
( \nu_2 (c_4(E)), \nu_2 (c_6(E)), \nu_2 (\Delta(E)))=(4, 6, 8), 
$$
we further check that $\Delta (E)_2 \equiv 1 \mod{4}$ and hence we also have
$e(E, 2)=24$.
We may therefore, in all cases, apply \cite[Theorem~4]{Fr} to find that,
 if $\nu_2 (\Delta (E)) \in \{ 4, 10 \}$, then $\phi$ is always symplectic, while,  if $\nu_2 (\Delta (E)) =8$, then
$$
\phi  \; \text{is symplectic} \quad \iff  \quad 2 \text{ is a square mod } p.
$$

\vskip2ex

\subsubsection{Applying the criteria at $\ell=3$}  
If $E=72q.1.a1$, $72q.2.a1$, $72q.3.a1$, $72q.7.a1$, $72q.8.a2$ or $72q.9.a2$,  then $E$ has potentially multiplicative reduction at $3$ and so, after a suitable quadratic twist (denoted $E_t$) the reduction 
becomes multiplicative and $\nu_3(\Delta (E_t)) = b$ or $2b$. 
Therefore, $3 \mid C$ and the twisted Frey curve $F_t$ must also have 
multiplicative reduction at $3$ 
and satisfy $\nu_3(\Delta (F_t)) = 2p \nu_3(C)-3$.
Since $p \nmid \nu_3(\Delta (F_t))$, it follows from \cite[Proposition~2]{KO} 
that $p \nmid b$ and
$$
\phi \; \text{is symplectic} \quad \iff  \quad -3b \text{ is a square mod } p,
$$
for $E=72q.7.a1$, $72q.8.a2$  and $72q.9.a2$, while
$$
\phi  \; \text{is symplectic} \quad \iff  \quad -6b \text{ is a square mod } p,
$$
for $E=72q.1.a1$, $72q.2.a1,$ and $72q.3.a1$.

\smallskip

For the curves $E=72q.4.a2$, $72q.4.b2$, $72q.5.a2$, $72q.5.b2$, $72q.6.a1$ or $72q.6.b1$,   
the reduction at $\ell=3$ is potentially good and tame 
(because $\nu_3(N_E) = 2$) and since $\nu_3(\Delta(E)) \ne 6$ we have $e(E,3)=4$.
As before, it follows that $e(F,3)=4$ (so that $3 \nmid C$), and we may apply \cite[Theorem~1]{FrKr2}. 
Let $r$ and $t$ be as in that theorem. In all cases we have $t=0$; 
furthermore, we have $r=0$ for $E=72q.4.a2$, $E=72q.5.a2$ or $E=72q.6.a1$, and $r=1$ for $E=72q.4.b2$, $E=72q.5.b2$  or $E=72q.6.b1$. It follows that $\phi$ is always symplectic in the first cases, while
$$
\phi \; \text{is symplectic} \quad \iff  \quad 3 \text{ is a square mod } p,
$$
in the latter three.

\subsubsection{Conclusions for level $72q$} \label{fool72}
From the calculations above we extract the following relations. 
For $E=72q.1.a1$, or either of $E=72q.2.a1$ or $E=72q.3.a1$ with $a=2$, it follows that
$$
\left( \frac{\alpha}{p} \right)  = \left( \frac{2}{p} \right)  =  \left( \frac{-6b}{p} \right),
$$
while, for $E=72q.2.a1$  or $E=72q.3.a1$ with $a=3$, 
$$
\left( \frac{\alpha}{p} \right)   =  \left( \frac{-6b}{p} \right) =1.
$$
If $E=72q.4.a2$, we have 
$$
\left( \frac{\alpha}{p} \right)   =  \left( \frac{2}{p} \right) =1,
$$
while $E=72q.5.a2$  or $E=72q.6.a1$ give 
$$
\left( \frac{\alpha}{p} \right)  =1.
$$
Taking $E=72q.4.b2$  yields
$$
\left( \frac{\alpha}{p} \right)  = \left( \frac{2}{p} \right)  =  \left( \frac{3}{p} \right),
$$
while $E=72q.5.b2$  or $E=72q.6.b1$  give
$$
\left( \frac{\alpha}{p} \right)   =  \left( \frac{3}{p} \right) =1.
$$
If $E=72q.7.a1$ or $72q.9.a2$, we have
$$
\left( \frac{\alpha}{p} \right)   =  \left( \frac{-3b}{p} \right) =1.
$$
Finally, if $E=72q.8.a2$, 
$$
\left( \frac{\alpha}{p} \right)  = \left( \frac{2}{p} \right)  =  \left( \frac{-3b}{p} \right).
$$

\section{Some applications of symplectic criteria} \label{applications}

As the preceding section reveals, there are many results we could state now for the various families of primes $S_i$ comprising the set $S_0$. For simplicity, we limit ourselves to the three statements we have mentioned in our introduction (Theorems \ref{shark99}, \ref{main3} and \ref{main4}) and one result valid for small values of $q$ (Theorem \ref{rump}). 

\medskip

\subsection{Proof of Theorem~\ref{shark99}}
If $q \not\in S_0$, the desired conclusion is immediate from Theorem \ref{main1}. Suppose, then, that $q \in S_0 \setminus T$ and that there exists a solution to (\ref{gump}) in coprime nonzero integers $A, B$ and $C$ and prime $p \geq q^{2q}$. In particular, we note, without further mention, that the primes $p$ under consideration all satisfy $\gcd (p,6q)=1$. Also, we have that $p \nmid (4-a) b$, whenever these parameters appear in the sequel.
From Section~\ref{Frey} and Theorem~\ref{main1}, it follows there exists an isomorphism
$\phi : F[p] \to E[p]$, where $F$ is the Frey-Hellegouarch curve and $E$ is one of the curves 
in Corollary~\ref{wonder}. Since $\alpha = 1$, we see from Proposition~\ref{multiplicativeAtq}
that $\phi$ is symplectic. Furthermore, the shape of the primes in $S_7$ 
implies that $7, 19 \in S_7$ and $E$ does not correspond to the isogeny classes
$36.q.2.a$, $36.q.2.b$, $72.q.5.a$ or $72.q.5.b$. In conclusion,  
we need to consider $E$ in the remaining conjugacy classes; in particular, we 
can either take $E$  isogenous to one of
$$
90c, \; 306c, \; 360a, \; 360d, \; 936d \; \mbox{ or } \; 5256e,
$$
whereby $q \in \{ 5, 13, 17, 73 \}$, or $E$ isomorphic 
to one of the following curves: 
$$
\begin{array}{c}
18q.1.a1, 18q.2.a1, 18q.3.a1, 18q.4.a2 \; (\mbox{with } \delta_1=\delta_2=0, a \mbox{ even}, b \mbox{ odd}),  18.q.5.a2  \mbox{ or } \\
18q.5.b2 \; (\mbox{with, in both cases, } \delta_1=\delta_2=0 \mbox{ and } a \mbox{ even}), 36q.1.a2, 36q.1.b2,  36q.3.a1,  \\
 72q.1.a1, 72q.2.a1, 72q.3.a1, 72.q.4.a2, 72q.4.b2, 72q.7.a1, 72q.8.a2 
\mbox{ or }  \\ 72q.9.a2 \; (\mbox{with } \delta_1=\delta_2=0, a =4).\\
\end{array}
$$
For $q  \leq 73$,  the desired conclusion will follow immediately from our Theorem \ref{rump}, which we will prove later in this section.
For the remaining  possible types for $q$, we will place a number of conditions upon $p$ to guarantee that, in each case, $\phi$ is anti-symplectic, providing the desired contradiction. These conditions will be of the form 
$\left( \frac{\kappa_i}{p} \right)   = -1$, for, in each case, a finite collection of  integers $\kappa_i$, and hence are each equivalent to $p$ lying in certain residue classes modulo $8 |\kappa_i|$. We remind the reader that a given prime $q$ has at most finitely many (isogeny classes of) curves $E$ associated to it. This will prove Theorem~\ref{shark99} provided we can show that these conditions are compatible, i.e. that we do not have three distinct indices $i$, say $i=1, 2$ and $3$, with $\kappa_1 \kappa_2 \kappa_3$ an integer square. In particular, compatibility is immediate if we have $\kappa_i$ negative for each $i$. Our goal will be to show that, for a given prime in $q \in S_0 \setminus T$, we can always find a corresponding set of $\kappa_i$ with either

\hskip3ex (i) $\kappa_i$ negative for all $i$, or 

\hskip3ex (ii) $\kappa_i$ either positive and $\kappa_i \equiv 2 \mod{4}$, or $\kappa_i$ negative  and odd, or

\hskip3ex (iii) $\kappa_i \equiv 2 \mod{4}$ for all $i$.

Combining the conclusions of subsections \ref{fool18}, \ref{fool36} and \ref{fool72}, we can choose $\kappa_i$ for which we require $\left( \frac{\kappa_i}{p} \right)   = -1$, to contradict the fact that $\phi$ is symplectic, as follows.
$$
\begin{array}{cc|cc}
E & \kappa_i & E & \kappa_i  \\ \hline
18q.1.a1 &  4-a \mbox{ or } -6b & 72q.1.a1 & 2 \mbox{ or } -6b \\ 
18q.2.a1 &  4-a \mbox{ or } -6 & 72q.2.a1 & -6b \\
18q.3.a1 & 4-a \mbox{ or } -6b & 72q.3.a1 &  -6b\\
18q.4.a2  & 12-2a \mbox{ or } -3b & 72q.4.a2 &  2 \\
18q.5.a2 &  12-2a & 72q.4.b2  & 2 \mbox{ or }  3 \\
18q.5.b2  &  12-2a & 72q.7.a1 & -3b \\
36q.1.a2 & 2 & 72q.8.a2  &  2 \mbox{ or } -3b \\
36q.1.b2 & 6 & 72q.9.a2  &  -3b \\
36q.3.a1 & -3b &  &  \\
\end{array}
$$
Here, the integers $a$ and $b$ are as given in the definitions of the curves $E$ in Section \ref{Class}. It is important to remember that, for a given $q$ and corresponding type of curve $E$, we have not ruled out the possibility of there being more than one non-isogenous curve involved. As example  (\ref{ex99}) illustrates, there can certainly be non-isogenous curves associated to a fixed pair $(q,E)$; in the case of (\ref{ex99}), neither curve of the shape $E=18q.4.a$ satisfies $a \equiv 0 \mod{2}, b \equiv 1 \mod{2}$.

From the preceding table,  the only cases where we cannot choose $\kappa_i$ to be negative are the primes $q$ corresponding to $E=36q.1.a2$, $36q.1.b2$, $72q.4.a2$ or $72q.4.b2$. The first two of these require $q \in S_8$, while the latter two arise from $q \in S_5$ of the form $q=3d^2+4$ for integer $d$. In each of these cases, we can choose $\kappa_i \equiv 2 \mod{4}$ positive (see the table above). 

To conclude the proof of Theorem \ref{shark99}, we need to show that for each $q$ which can possibly correspond to any of 
$$
E \in E_0 = \left\{ 36q.1.a2, \; 36q.1.b2, \; 72q.4.a2\, \; 72q.4.b2 \right\},
$$
if we have a solution to (\ref{gump})  that  is  associated to some $E \not\in E_0$, then we can eliminate a positive proportion of prime exponent $p$ by requiring that $\left( \frac{\kappa_i}{p} \right)   = -1$ for $\kappa_i$ either positive and $\kappa_i \equiv 2 \mod{4}$, or $\kappa_i$ negative  and odd (as in case (ii) of the preceding page), or all $\kappa_i \equiv 2 \mod{4}$ (as in case (iii)). In particular, we need to start by understanding $S_8 \cap S_i$ and $\{ q  \mbox{ prime} : q=3d^2+4 \} \cap S_i$; Proposition \ref{intersect} is a good place to begin.

 Suppose first that  $q \in S_5$ and $q = 3d^2+4$ for an odd integer $d$. Modulo $8$, we cannot have $q=3d_1^2 + 2^\alpha$ for integer $d_1$ and $\alpha \geq 4$. Further, applying Proposition \ref{intersect}, we have that $q \not\in S_i$ for each $i \in \{ 1, 3, 4, 8 \}$.  Since $q \in S_0 \setminus T$, we have that $q \not\in S_7$. If $q \in S_2$, then $q = (-1)^{\delta_1} 2^{a_2} + (-1)^{\delta_2}  3^{b_2}$, for $a_2 \in \{ 2, 3 \}$ or $a_2 \geq 5$,  and $b_2 \geq 1$. Modulo $3$, $\delta_1 \equiv a_2 \mod{2}$, while, modulo $8$, either $a_2=2$, $\delta_1=\delta_2=0$, $b_2 \equiv 1 \mod{2}$ and $d = 3^{(b_2-1)/2}$, or we have $a_2 \geq 3$, $b_2 \equiv 0 \mod{2}$ and $\delta_2=1$. In this latter case, we also have $\delta_1=0$, $a_2 \equiv 0 \mod{2}$ and hence 
$$
2^{a_2/2}-3^{b_2/2} = 1 \; \; \mbox{ and } \; \; 2^{a_2/2}+3^{b_2/2} = q.
$$
Since $a_2 \geq 6$, this first equation has no solutions (by an old result of Levi ben Gerson), whereby it follows that if $q = 3d^2+4 \in S_2$, then $d = 3^{(b_2-1)/2}$. If, further,  $q = 3d^2+4 \in S_2 \cap S_6$, then
$$
q=3^{b_2}+4 = \frac{d_6^2 + 3^{b_6}}{4},
$$
for odd positive integers $b_2, d_6$ and $b_6$, so that
\begin{equation} \label{old}
d_6^2 = 4 \cdot 3^{b_2} - 3^{b_6} +16.
\end{equation}
In general, this equation has precisely the solutions 
$$
(d_6,b_2,b_6) = (1,1,3),  \, (5,1,1), \,  (11,3,1) \; \mbox{ and } \; (31,5,3)
$$
in odd positive integers; none of these correspond to a prime values of $q > 73$. To prove this, note that an elementary argument easily yields that $b_2 > b_6$ unless $|d_6| \leq 5$. We may thus write $d_6 = 3^{b_6} \cdot k_1 + (-1)^\delta 4$, for some $\delta \in \{ 0, 1 \}$ and $k_1 \equiv \pm 1 \mod{6}$ a positive integer. Substituting into (\ref{old}), we have
$$
3^{b_6} k_1^2 + (-1)^\delta 8 \cdot k_1 = 4 \cdot 3^{b_2-b_6} - 1.
$$
If $k_1=1$, then, modulo $3$, we have $3^{b_6-1} +  3  = 4 \cdot 3^{b_2-b_6-1}$, corresponding to $(d_6,b_2,b_6) =(31,5,3)$. If $k_1 > 1$ then $k_1 \geq 5$ and necessarily $b_2 > 2 b_6$. It follows that we can write $(-1)^\delta 8 \cdot k_1+1 = 3^{b_6} \cdot k_2$ for a (nonzero) integer $k_2 \equiv 3 \mod{8}$, so that
\begin{equation} \label{healthy}
\left( 3^{b_6} k_2 -1 \right)^2 + 64 k_2  = 256 \cdot 3^{b_2-2 b_6}.
\end{equation}
We check that the only solution to this equation with $k_2 \in \{ -13, -5, 3, 11 \}$ corresponds to $(d_6,b_2,b_6) = (11,3,1)$; otherwise, after a little work, we may suppose that $b_2 > 4 b_6$ and hence that
$-2 k_2 3^{b_6} + 64 k_2 +1$ is divisible by $3^{2 b_6}$ (and hence $|2 k_2 3^{b_6} - 64 k_2 -1| \geq 3^{2b_6}$). It follows that either $b_6 \leq 3$, or that we have $|k_2| > 3^{b_6-1}$. From (\ref{healthy}), after a little more work, we may thus conclude that either $b_6 \in  \{ 1, 3 \}$, or that $b_2 \geq 6 b_6 -7$.

On the other hand, applying Theorem 1.5 of \cite{BaBe}, with (in the notation of that theorem)
$$
(a,y,x_0,m_0,\Delta,\alpha,s) = (1, 3, 3788, 15, -37, 3.1, 2),
$$
we find that
$$
\left| \sqrt{3} - \frac{p}{2 \cdot 3^k} \right| > e^{-170} \, 3^{-1.64281 k},
$$
for $p$ and $k$ positive integers with $k \geq 4775$. It follows that
$$
\left| p^2 - 4\cdot 3^{2k+1} \right| >  4 \cdot e^{-170} \; 3^{0.35719 k},
$$
provided $k \geq 4775$. Applying this with $p=d_6$ and $b_2=2k+1$, (\ref{old}) thus implies that either $b_6 \in  \{ 1, 3 \}$ or we have
$$
3^{(b_2+7)/6} \geq 3^{b_6} > 4 \cdot e^{-170} \; 3^{0.35719 (b_2-1)/2},
$$
whence $b_2 \leq 12979$. A brute-force search confirms that (\ref{old}) has only the listed solutions.

In conclusion, then, if $q=3d^2+4$ for integer $d$, we have one of the following situations. Either $q \not \in S_i$ for $i \neq 5$ and hence corresponds to only $E=72q.4.a2$ and $E=72q.4.b2$, or $q \in S_2 \cap S_5$ corresponds to precisely  
$$
E \in \left\{ 72q.3.a1, 72q.4.a2, 72q.4.b2 \right\},
$$
or $q \in S_5 \cap S_6$, with 
$$
E \in \left\{ 36q.3.a1, 72q.4.a2, 72q.4.b2 \right\},
$$
for, possibly, several different curves of the shape $36q.3.a1$, depending on the number of distinct ways to represent $q =  \frac{d_{6,i}^2+3^{b_i}}{4}$, for integers $d_{6,i}$ and $b_i$.
In the first case, we deduce a contradiction for every $p$ with $\left( \frac{2}{p} \right)   = -1$. In the second, we require
$$
\left( \frac{2}{p} \right)   = \left( \frac{-6b}{p} \right)   = -1,
$$
where $q=3^b  +4$ (and $b$ is odd). Finally, for the third case, we suppose that
$$
\left( \frac{2}{p} \right)   = \left( \frac{-3b_1}{p} \right)   = \cdots  =  \left( \frac{-3b_t}{p} \right)  =  -1,
$$
where 
$$
q = 3d^2+4 =  \frac{d_{6,i}^2+3^{b_i}}{4}, \; \mbox{ for }  i=1, \ldots, t.
$$
Here, the exponents $b_i$ are necessarily odd. In each case, these choices for $p$ contradict the fact that the $p$-torsion of our Frey-Hellegouarch curve $F$ is symplectically isomorphic to that of  $E$.

Let us next suppose that $q \in S_8$ with $q > 73$. From Proposition \ref{intersect}, it follows that $q \not\in S_i$, for $i \in \{ 1, 3, 5, 7 \}$.
By definition, there exist integers $u$ and $v$ such that  $q= (3v^2-1)/2$ where $u^2-3v^2=-2$. The positive integers $v = v_k$ satisfying this latter equation also satisfy the binary recurrence 
\begin{equation} \label{recur}
v_{k+1} = 4 v_k - v_{k-1}, \; \mbox{ for } \; k \geq 1, \; \mbox{ where } \; v_0=1, v_1 = 3.
\end{equation}
In particular, we have that $v_k \equiv 0 \mod{3}$ precisely when $k \equiv 1 \mod{4}$. For such $k$, we may readily show via induction that $v_k \equiv \pm 3 \mod{13}$ and hence that $3 v_k^2-1 \equiv 0 \mod{13}$. It follows that, in order to have $q=(3v^2-1)/2$ prime with $u^2-3v^2=-2$ for some integer $u$, we require that either $q=13$, or that $v \equiv \pm 1 \mod{3}$ (whereby $q \equiv 1 \mod{9}$).  If, for our $q \in S_8$ with $q > 73$, we have $q =2^a+3^b$ for integers $a \geq 2$ and $b \geq 1$, then, modulo $4$, $b$ is necessarily even, so that we require $2^a \equiv 1 \mod{9}$, whence $a \equiv 0 \mod{6}$. It follows that $2^a+3^b \equiv 2, 3$ or $5 \mod{7}$. On the other hand, again from considering the recursion (\ref{recur}), we find that $q \equiv \pm 1 \mod{7}$, a contradiction. If, instead, we have $q=2^a -3^b$, for $a \geq 2$ and $b \geq 1$, then, modulo $12$, $a$ is even and $b$ is odd. If $b=1$, then we have $2^{a+1} = 3 v^2 +5$ and so, from Lemma \ref{frosty}, since $q > 73$, a contradiction. If we suppose that $b \geq 2$, then, modulo $9$, we again require that $a \equiv 0 \mod{6}$, so that
$2^a - 3^b \equiv 0, 3, 5, 9$ or $11 \mod{13}$. On the other hand, from (\ref{recur}), we have that $q \equiv \pm 1 \mod{13}$, a contradiction.

It follows that,  if $q \in S_2 \cap S_8$ with $q > 73$, then there exist  integers $a \geq 2$ and $b \geq 1$, with $q= 3^b-2^a$. Arguing as previously, modulo $2^2 \cdot 3^3 \cdot 7$, we find that $a \equiv 3 \mod{6}$ and $b \equiv 2 \mod{6}$. Working modulo $73$, we find from (\ref{recur}) that $q \equiv \pm 1, \pm 34, \pm 35 \mod{73}$ which shows that, in fact, $b \equiv 2 \mod{12}$. 

We thus have, for $q \in S_8$, that  $E$ is necessarily one of
$$
18q.3.a1, 18q.4.a2, 36q.1.a2, 36q.1.b2, 72q.3.a1, 72q.7.a1 \mbox{ or } 72q.9.a2,
$$
where $E=18q.3.a1$ and $72q.3.a1$ correspond to $q \in S_2$, $E=18q.4.a2$ and $72q.9.a2$ come from $q \in S_4$, $E=72q.7.a1$ arises from $q \in S_6$ and both $E=36q.1.a2$ and $36q.1.b2$ occur for every $q \in S_8$.
From our previous discussion, if $E=18q.3.a1$, there exist integers $a_2 \geq 5$ and $b_2 \geq 3$ with  $a_2 \equiv 3 \mod{6}$, $b_2 \equiv 2 \mod{12}$ and $q = 3^{b_2}-2^{a_2}$. If we choose $p$ such that $\left( \frac{4-a_2}{p} \right) =-1$ (where $4-a_2$ is negative and odd), we again contradict the fact that the $p$-torsion of our Frey-Hellegouarch curve $F$ is symplectically isomorphic to that of $E$. If $E=18q.4.a2$, we have that $q=d_4^2 + 2^{a_4} 3^{b_4}$ for integers $d_4, a_4$ and $b_4$ (not necessarily unique), with $a_4 \geq 8$ even and $b_4$ odd. In this case, we constrain $p$ to satisfy $\left( \frac{-3 b_4}{p} \right)   =-1$. Similarly, for $E =36q.1.a2$ or $36q.1.b2$, we impose the conditions
$$
\left( \frac{2}{p} \right)   = \left( \frac{6}{p} \right)   = -1,
$$
while, for $E=72q.3.a1$, we have  $q = 3^{b_2}-8$ with $b_2 \equiv 2 \mod{12}$, and choose 
$$
\left( \frac{-6 b_2}{p} \right)   =\left( \frac{-3 b_2/2}{p} \right) = -1;
$$
note that, importantly for us, $-3b_2/2$ is odd. If $E=72q.7.a1$, we can write $q=\frac{d^2 + 3^{b_6}}{4}$ with $b_6$ odd and take $\left( \frac{-3 b_6}{p} \right)   =-1$. Finally, if $E=72q.9.a2$, we have $q=d_4^2 + 2^{4} 3^{b_4}$ for integers $d_4$ and $b_4$ (not necessarily unique), with  $b_4$ odd, and can derive a contradiction by choosing $p$ such that $\left( \frac{-3 b_4}{p} \right)   =-1$. 
In summary, if $q \in S_8$, we reach our desired conclusion by choosing our finite set of $\kappa_i$ as in case (ii).
This completes the proof of Theorem \ref{shark99}.

\subsection{Proof of Theorem~\ref{main3}} 
Let $q=2^a 3^b -1$ with $a \geq 5$ and $b \geq 1$ be a prime. Then $q \in S_1$ and hence, from Proposition \ref{intersect}, $q \not\in S_i$ for $i \in \{ 2, 3, 5, 7, 8 \}$. On the other hand, $q \not\in S_i$ for $i \in \{ 4, 6 \}$, since $q \equiv 2 \mod{3}$. It follows that, in this case, a solution to (\ref{ABC}) with $p \geq q^{2q}$ necessarily corresponds to an elliptic curve in the isogeny class $18q.1.a$. The result now follows from the equalities in \eqref{sample}.

\subsection{Proof of Theorem~\ref{main4}}  
Suppose that $A, B$ and $C$ are coprime, nonzero integers satisfying \eqref{quiet} with $p \geq 17$, 
and write $F$ for the 
corresponding Frey-Hellegouarch curve. 
Note that, for $q=5$, we are led to consider  levels $90, 180$ and $360$. For these levels,
each weight $2$, cuspidal newform $f$ corresponds to one of the $9$
isogeny classes of elliptic curves $E/\Q$ given in Cremona's notation by
$$
 90a, \; 90b, \; 90c, \; 180a, \; 360a, \; 360b, \; 360c, \; 360d \; \text{ and } \; 360e.
$$
For $E$ in the isogeny classes $90a, 90b, 180a, 360b$ and $360c$, we find that 
$a_7(E)=2$ and hence, it follows from \eqref{shark}, the Hasse bound and 
the level lowering condition that 
$$
2 \equiv 0, \pm 4, \pm 8 \mod{p}. 
$$
This gives a contradiction with $p \geq 17$.

Next, we treat the isogeny class $360e$. Taking $E=360e2$, we find that $e(E,3)=2$. 
In the beginning of Section~\ref{inertia4At3}, it is explained that either $F$ has 
potentially multiplicative reduction at $\ell=3$ or potentially good reduction with $e(F,3) = 4$, 
 a contradiction in either cases.

Finally, suppose that $E$ is in one of the isogeny classes $90c$, $360a$ and $360d$, say, $E=90c2, 360a2$ or $360d2$. We will apply \cite[Theorem~4]{Fr} and \cite[Proposition~2]{KO} with $\ell \in \{ 2, 3, q \}$. In all cases, from \cite[Proposition~2]{KO} with $\ell=q$, we have that our isomorphism between  $F[p]$ and $E[p]$ is necessarily symplectic.
If $E=90c2$, we may thus further appeal to \cite[Proposition~2]{KO} with  $\ell=2$ and $\ell=3$ 
(after suitable twist) to conclude that
\begin{equation} \label{test1}
\left( \frac{-1}{p} \right)   =  \left( \frac{-2}{p} \right) =1.
\end{equation}
For $E=360a2$, we apply \cite[Theorem~4]{Fr} and \cite[Proposition~2]{KO} with $\ell=3$, whereby
\begin{equation} \label{test2}
\left( \frac{2}{p} \right)   = \left( \frac{-3}{p} \right)   =1.
\end{equation}
If $E=360d2$, we apply \cite[Proposition~2]{KO} with $\ell =3$ to conclude that
\begin{equation} \label{test3}
\left( \frac{-6}{p} \right)   =1.
\end{equation}
We reach our desired conclusion upon observing that, if $p \equiv 13, 19$ or $23 \mod{24}$, then each of (\ref{test1}), (\ref{test2}) and (\ref{test3}) fails to hold.

\subsection{Further results for small primes $q$} \label{last}
To conclude this paper, we will provide some more explicit results for small values of $q$.
We obtain these by proceeding in a similar fashion to the proof of Theorem~\ref{main4}. Making the further assumption that $p \geq q^{2q}$, we reduce the calculation
to consideration of 
elliptic curves $E$ with non-trivial rational $2$-torsion, conductor in the set $\{ 18q, 36q, 72q \}$ and
such that $\Delta(E)$ is of the shape $T^2$ or $-3T^2$ for some integer $T$ (i.e. those corresponding to primes in $S_0$).
We summarize our results as follows.
\begin{theorem} \label{rump}
If $p$ and $q$ are primes with $p \geq q^{2q}$, then there are no coprime, nonzero integers $A, B$ and $C$ satisfying equation $\eqref{gump}$ with $q$ in the following table and $p$ satisfying the listed conditions.
\begin{small}
$$
\begin{array}{cc|cc}
q & p & q & p \\ \hline
5 & 13, 19, 23 \mod{24} & 47 & 5, 11, 13, 17, 19, 23 \mod{24} \\
11 & 13, 17, 19, 23 \mod{24} & 59 & 5, 7, 11, 13, 19, 23 \mod{24} \\
13 & 11 \mod{12} & 67 &  7, 11, 13, 29, 37, 41, 43, 59, 67, 71, 89, 101, 103 \mod{120}  \\
17 & 5, 17, 23 \mod{24} & 71 &  5 \mod{6} \\
23 & 19, 23 \mod{24} & 73 &  41, 71, 89 \mod{120} \\
29 & 7, 11, 13, 17, 19, 23 \mod{24} & 79 &  5, 7, 11, 13, 19, 23 \mod{24} \\
31 & 5, 11 \mod{24} & 89 & 13, 17, 19, 23  \mod{24} \\
41 & 5, 7, 11, 17,19,  23 \mod{24} & 97 &  11 \mod{12} \\
\end{array}
$$
\end{small}
\end{theorem}

Here, we have omitted both primes for which Theorem~\ref{main1} applies directly (i.e. $q = 53$ and $83$, according to Corollary~\ref{cor-blimey}) and also primes for which the symplectic method fails to eliminate exponents, i.e. $q \in \{ 7, 19, 37, 43, 61 \}$.
For these latter primes, observe that, in each case,  $q$  is of the shape $(3d^2+1)/4$ or $3d^2+16$ for an integer $d$; as explained in Section~\ref{trivial}, these are those primes for which there exists a solution to (\ref{gump}) (with $C=1$) for every exponent $p$ (whereby we expect our techniques to fail), together with those for which we have a ``trivial'' solution to the related equation $A^3+B^3=8q C^p$, again for every $p$.

\section{Appendix : $c$-invariants} \label{appendix}

\subsection{Conductor $18q$}
We have
\begin{small}
$$
\begin{array}{|cc|cc|} \hline
\mbox{ Curve} & c_4  & \mbox{ Curve} & c_4 \\ \hline
18q.1.a1 & 3^2 \cdot  \left( 2^{2a} 3^{2b} + (-1)^\delta 2^{a} 3^{b} + 1 \right) & 18q.3.a3 & 3^2 \cdot  \left( 2^{2a-4}  +  (-1)^{\delta_1+\delta_2} 2^{a}3^b + 3^{2b} \right) \\
18q.1.a2 & 3^2 \cdot  \left( 2^{2a+4} 3^{2b} + (-1)^{\delta} 2^{a+4} 3^{b} + 1 \right)  & 18q.3.a4 & 3^2 \cdot  \left( 2^{2a}  + (-1)^{1+\delta_1+\delta_2} 7 \cdot 2^{a+1}3^b + 3^{2b} \right) \\ 
18q.1.a3 & 3^2 \cdot  \left( 2^{2a-4} 3^{2b} + (-1)^{\delta} 2^{a} 3^{b} + 1 \right) & 18q.4.a1 & (-1)^{\delta_1} 3^2 \left( q - (-1)^{\delta_2} 2^{a-2} 3^b \right) \\
18q.1.a4 & 3^2 \cdot  \left( 2^{2a} 3^{2b} + (-1)^{\delta+1} 7 \cdot  2^{a+1} 3^{b} + 1 \right) & 18q.4.a2 &  (-1)^{\delta_1} 3^2 \left( q - (-1)^{\delta_2} 2^{a+2} 3^b \right) \\ 
18q.2.a1 & 3^2 \cdot  \left( 2^{2a}  + 2^{a} + 1 \right) & 18q.5.a1& 3^2 \left( d^2 + (-1)^{\delta_1+\delta_2} 2^{a-2} \right) \\
18q.2.a2 & 3^2 \cdot  \left( 2^{2a+4} + 2^{a+4} + 1 \right)  & 18q.5.a2 & 3^2 \left( d^2 - (-1)^{\delta_1+\delta_2} 2^{a} \right) \\
18q.2.a3 & 3^2 \cdot  \left( 2^{2a-4} + 2^{a} + 1 \right) & 18q.5.b1 & 3^4 \left( d^2 + (-1)^{\delta_1+\delta_2} 2^{a-2} \right) \\
18q.2.a4 & 3^2 \cdot  \left( 2^{2a} -7 \cdot 2^{a+1} + 1 \right) & 18q.5.b2 &  3^4 \left( d^2 - (-1)^{\delta_1+\delta_2} 2^{a} \right) \\ 
18q.3.a1 & 3^2 \cdot  \left( 2^{2a}  + (-1)^{\delta_1+\delta_2} 2^{a}3^b + 3^{2b} \right) & 18q.6.a1 & 3^2 \left( d^2 + 3 \cdot 2^{a-2} \right) \\
18q.3.a2 & 3^2 \cdot  \left( 2^{2a+4}  +  (-1)^{\delta_1+\delta_2} 2^{a+4}3^b + 3^{2b} \right) & 18q.6.a2 & 3^2 \left(d^2 - 3 \cdot 2^a \right) \\
 \hline
\end{array}
$$
\end{small}

$$
\begin{array}{|cc|} \hline
\mbox{ Curve} & c_6  \\ \hline
18q.1.a1 &  (-1)^{\delta+1} 3^3 \left( 2^{a+1} 3^b + (-1)^{\delta} \right)  \left( 2^{a} 3^b+(-1)^{\delta+1} \right) \left( 2^{a-1} 3^b + (-1)^{\delta} \right) \\
18q.1.a2 & 3^3 \left( 2^{a+1} 3^b + (-1)^{\delta} \right) \left( 2^{2a+5} 3^{2b} + (-1)^{\delta} 2^{a+5} 3^b -1 \right)  \\ 
18q.1.a3 & (-1)^{\delta+1} 3^3 \left( 2^{a-1} 3^b + (-1)^{\delta} \right)  \left( 2^{2a-5} 3^{2b} +(-1)^{\delta+1}  2^a 3^b - 1 \right)  \\
18q.1.a4 & (-1)^{\delta+1} 3^3 \left( 2^a3^b+(-1)^{\delta+1} \right) \left( 2^{2a}3^{2b} +(-1)^{\delta} 17 \cdot 2^{a+1}3^b+1 \right) \\ 
18q.2.a1 & -3^3 \left( 2^{a+1} +1 \right) \left( 2^{a} -1 \right) \left( 2^{a-1} +1 \right)  \\
18q.2.a2 & -3^3 \left( 2^{a+1} + 1 \right) \left( 2^{2a+5} +  2^{a+5}  -1 \right)   \\
18q.2.a3 & -3^3 \left( 2^{a-1} + 1 \right) \left( 2^{2a-5} -  2^{a}  -1 \right)  \\
18q.2.a4 & -3^3 \left( 2^{a} - 1 \right) \left( 2^{2a} +  17 \cdot 2^{a+1}  +1 \right)  \\ 
18q.3.a1 & (-1)^{\delta} 3^3 \left( 2^{a+1} +  (-1)^{\delta_1+\delta_2} 3^b \right) \left( 2^{a} -  (-1)^{\delta_1+\delta_2} 3^b \right) \left( 2^{a-1} +  (-1)^{\delta_1+\delta_2} 3^b \right)  \\
18q.3.a2 & (-1)^{\delta} 3^3 \left( 2^{a+1} +(-1)^{\delta_1+\delta_2} 3^b \right) \left( 2^{2a+5} +(-1)^{\delta_1+\delta_2} 2^{a+5} 3^b-3^{2b} \right)  \\
18q.3.a3 & (-1)^{b+1} 3^3 \left( (-1)^{\delta_1+\delta_2} 2^{a-1} + 3^b \right) \left( 2^{2a-5}- (-1)^{\delta_1+\delta_2} 2^a 3^b - 3^{2b} \right) \\
18q.3.a4 &  (-1)^{b} 3^3 \left( 3^b +(-1)^{b+\delta} 2^a \right) \left( 2^{2a} + (-1)^{\delta_1+\delta_2} 17 \cdot 2^{a+1} 3^b+3^{2b} \right) \\
18q.4.a1 & 3^3 d \left( d^2+ (-1)^{\delta_1+\delta_2} 2^{a-3} 3^{b+2}  \right) \\
18q.4.a2 & 3^3 d \left( d^2+(-1)^{\delta_1+\delta_2} 2^{a} 3^{b+2}  \right) \\
18q.5.a1 & 3^3 d \left( d^2+ (-1)^{\delta_1+\delta_2} 2^{a-3} 3  \right) \\
18q.5.a2 & 3^3 d \left( d^2+(-1)^{\delta_1+\delta_2} 2^{a} 3  \right)  \\
18q.5.b1 & -3^6 d \left( d^2+ (-1)^{\delta_1+\delta_2} 2^{a-3} 3  \right) \\
18q.5.b2 & -3^6 d \left( d^2+(-1)^{\delta_1+\delta_2} 2^{a} 3  \right) \\
18q.6.a1 & 3^3 d \left( d^2+ 2^{a-3} 3^{2}  \right)  \\
18q.6.a2 &  3^3 d \left( d^2+2^{a} 3^{2}  \right) \\
 \hline
\end{array}
$$

\subsection{Conductor $36q$}
We have

$$
\begin{array}{|cc|cc|} \hline
\mbox{ Curve} & c_4 & \mbox{ Curve} & c_4 \\ \hline
36q.1.a1 & 2^4 3 \left( q^2-1 \right) &  36q.4.a1 & 2^4 3^2 \left(d^2 - 3^{b+1} \right) \\
36q.1.a2 & 2^4 3 \left( 16q^2-1 \right)  &  36q.4.a2 &  2^4 3^2 \left( d^2 + 4 \cdot 3^{b+1} \right) \\ 
36q.1.b1 & 2^4 3^3 \left( q^2-1 \right) &  36q.5.a1 & 2^4 3^2 \left(d^2 -  3^{b+1} \right) \\
36q.1.b2 & 2^4 3^3 \left( 16q^2-1 \right) & 36q.5.a2 & 2^4 3^2 \left( d^2 +  4 \cdot 3^{b+1} \right) \\ 
36q.2.a1 & 2^2 3^2 \left( d^2-1 \right) & 36q.6.a1 & 2^4 3^2 \left( d^2-1 \right)   \\
36q.2.a2 & 2^4 3^2 \left( 4d^2+1 \right)  &  36q.6.a2 & 2^4 3^2 \left( d^2+4 \right) \\
36q.2.b1 & 2^2 3^4 \left( d^2-1 \right) &  36q.6.b1 & 2^4 3^4 \left( d^2-1 \right)  \\
36q.2.b2 & 2^4 3^4 \left( 4 d^2+1 \right) & 36q.6.b2 & 2^4 3^4 \left( d^2+4 \right) \\ 
36q.3.a1 & 2^2 3^2 \left(d^2 -  3^{b+1} \right) & 36q.7.a1 &  2^4 3^2 \left(d^2 + 3^{b+1} \right)  \\
36q.3.a2 & 2^4 3^2 \left( 4 d^2 + 3^{b+1} \right) & 36q.7.a2 & 2^4 3^2 \left( d^2 - 4 \cdot 3^{b+1} \right) \\
 \hline
\end{array}
$$

$$
\begin{array}{|cc|cc|} \hline
\mbox{ Curve} & c_6  & \mbox{ Curve} & c_6 \\ \hline
36q.1.a1 & -2^5 3^2 r s \left( q^2+2  \right) &  36q.4.a1 & 2^5 3^3 d \left( 2 d^2 - 3^{b+2}  \right) \\
36q.1.a2 &  -2^6 3^2 r s \left( 32 q^2+ 1  \right) &  36q.4.a2 &  2^6 3^3 d \left( d^2 - 4 \cdot  3^{b+2}  \right) \\ 
36q.1.b1 & 2^5 3^5 r s \left( q^2+2  \right) &  36q.5.a1 & 2^5 3^3 d \left( 2 d^2- 3^{b+2}  \right)  \\
36q.1.b2 & 2^6 3^5 r s \left( 32 q^2+ 1  \right) & 36q.5.a2 & 2^6 3^3 d \left( d^2- 4 \cdot  3^{b+2}  \right)   \\ 
36q.2.a1 & -2^3 3^3 d \left( d^2+3  \right) & 36q.6.a1 & 2^5 3^3 d \left( 2d^2-3  \right) \\
36q.2.a2 & -2^6 3^3 d \left( 8 d^2+ 3  \right)  & 36q.6.a2 & 2^6 3^3 d \left( d^2-12  \right) \\
36q.2.b1 & 2^3 3^6 d \left( d^2+3  \right) &  36q.6.b1 & -2^5 3^6 d \left( 2 d^2-3  \right) \\
36q.2.b2 & 2^6 3^6 d \left( 8 d^2+ 3  \right) & 36q.6.b2 & -2^6 3^6 d \left( d^2-12  \right)   \\ 
36q.3.a1 & -2^3 3^3 d \left( d^2 + 3^{b+2}  \right) & 36q.7.a1 & 2^5 3^3 d \left( 2 d^2 + 3^{b+2}  \right) \\
36q.3.a2 & -2^6 3^3 d \left( 8d^2+ 3^{b+2}  \right) & 36q.7.a2 & 2^6 3^3 d \left( d^2 + 4 \cdot  3^{b+2}  \right) \\
 \hline
\end{array}
$$

\subsection{Conductor $72q$}
We have
$$
\begin{array}{|cc|cc|} \hline
\mbox{ Curve} & c_4  & \mbox{ Curve} & c_4 \\ \hline
72q.1.a1 &  2^4 3^2 \left(3^{2b}+3^b+1 \right) & 72q.5.b1 & 2^4 3^4 \left( d^2 + (-1)^\delta 2^{a-2} \right) \\
72q.1.a2 & 2^4 3^2 \left( 16 \cdot 3^{2b}+ 16 \cdot 3^b+1 \right)  & 72q.5.b2 &  2^4 3^4 \left( d^2 - (-1)^\delta 2^{a} \right) \\
72q.1.a3 &  2^4 3^2 \left( 16 \cdot 3^{2b}+ 16 \cdot 3^b+1 \right) & 72q.6.a1 & 2^2 3^2 \left( d^2-1 \right)  \\
72q.1.a4 & 2^4 3^2 \left( 16 \cdot 3^{2b}+ 16 \cdot 3^b+1 \right) & 72q.6.a2 &  2^4 3^2 \left( 4 d^2+1 \right) \\ 
72q.2.a1 & 2^4 3^2 \left(2^{2a}3^{2b}+ (-1)^\delta 2^a 3^b+1 \right) & 72q.6.b1 & 2^2 3^4 \left( d^2-1 \right) \\
72q.2.a2 &  2^4 3^2 \left(2^{2a+4}3^{2b}+ (-1)^\delta 2^{a+4} 3^b+1 \right) & 72q.6.b2 & 2^4 3^4 \left( 4d^2+1 \right) \\ 
72q.2.a3 & 2^4 3^2 \left(2^{2a-4}3^{2b}+ (-1)^\delta 2^a 3^b+1 \right) & 72q.7.a1 & 2^2 3^2 \left( d^2-3^{b+1} \right) \\
72q.2.a4 &  2^4 3^2 \left(2^{2a}3^{2b}+ (-1)^\delta 7 \cdot 2^{a+1} 3^b+1 \right) & 72q.7.a2 & 2^4 3^2 \left( 4d^2-3^{b+1} \right) \\
72q.3.a1 &  2^4 3^2 \left(3^{2b}+ (-1)^\delta 2^a 3^b+2^{2a} \right) & 72q.8.a1 &  2^4 3^2 \left( d^2+3^{b+1} \right)  \\
72q.3.a2 & 2^4 3^2 \left(3^{2b}- (-1)^\delta 7 \cdot 2^{a+1} 3^b+2^{2a} \right) &  72q.8.a2 & 2^4 3^2 \left( d^2-4 \cdot 3^{b+1} \right) \\ 
72q.3.a3 & 2^4 3^2 \left(3^{2b}+ (-1)^\delta  2^{a+4} 3^b+2^{2a+4} \right), & 72q.9.a1 & 2^4 3^2 \left( d^2+(-1)^{\delta_1+\delta_2} 2^{a-2} 3^{b+1} \right) \\
72q.3.a4 & 2^4 3^2 \left(3^{2b}+ (-1)^\delta  2^{a} 3^b+2^{2a-4} \right)  & 72q.9.a2 & 2^4 3^2 \left( d^2+ (-1)^{\delta_1+\delta_2+1} 2^a \cdot 3^{b+1} \right) \\
72q.4.a1 & 2^4 3^2 \left( d^2 +  1 \right) & 72q.10.a1 & 2^4 3^2 \left( d^2 - 2^{a-2} 3^{b+1} \right) \\
72q.4.a2 &  2^4 3^2 \left( d^2 - 4 \right) & 72q.10.a2 & 2^4 3^2 \left( d^2+ 2^a \cdot 3^{b+1} \right) \\
72q.4.b1 & 2^4 3^4 \left( d^2 + 1 \right) & 72q.11.a1 & 2^4 3^2 \left( d^2 +  24 \right) \\
72q.4.b2 &  2^4 3^4 \left( d^2 - 4 \right) & 72q.11.a2 & 2^4 3^2 \left( d^2 - 96 \right) \\
72q.5.a1 & 2^4 3^2 \left( d^2 + (-1)^\delta 2^{a-2} \right) & 72q.12.a1 & 2^4 3^2 \left( d^2 +  24 \right) \\
72q.5.a2 &  2^4 3^2 \left( d^2 - (-1)^\delta 2^{a} \right) & 72q.12.a2 & 2^4 3^2 \left( d^2 - 96 \right) \\
 \hline
\end{array}
$$

$$
\begin{array}{|cc|} \hline
\mbox{ Curve} & c_6  \\ \hline
72q.1.a1 &  2^5 3^3 ( 3^b-1) (3^b+2) (2 \cdot 3^b+1)   \\
72q.1.a2 & 2^6 3^3 (2 \cdot 3^b+1) \left( 32 \cdot 3^{2b}+ 32 \cdot 3^b-1 \right)   \\ 
72q.1.a3 &  2^6 3^3 (2 \cdot 3^b+1) \left( 32 \cdot 3^{2b}+ 32 \cdot 3^b-1 \right)  \\
72q.1.a4 & 2^6 3^3 (2 \cdot 3^b+1) \left( 32 \cdot 3^{2b}+ 32 \cdot 3^b-1 \right)  \\ 
72q.2.a1 & -2^6 3^3  ((-1)^\delta 2^{a+1} 3^b+1)  \left( 2^{2a-1} 3^{2b} + (-1)^\delta 2^{a-1} 3^b - 1  \right)  \\
72q.2.a2 &  -2^6 3^3  ((-1)^\delta 2^{a+1} 3^b+1)  \left( 2^{2a+5} 3^{2b} + (-1)^\delta 2^{a+5} 3^b - 1  \right) \\
72q.2.a3 & 2^5 3^3  ((-1)^\delta 2^{a-1} 3^b+1)  \left( -2^{2a-4} 3^{2b} + (-1)^\delta 2^{a+1} 3^b +2  \right)  \\
72q.2.a4 & -2^6 3^3  ((-1)^\delta 2^{a} 3^b-1)  \left( 2^{2a} 3^{2b} + (-1)^\delta 17 \cdot 2^{a+1} 3^b +1  \right) \\ 
72q.3.a1 &  2^6 3^3  (-1)^b (3^b- (-1)^\delta 2^a)  \left( 2^{2a} + (-1)^\delta 5 \cdot 2^{a-1} 3^b + 3^{2b}  \right)  \\
72q.3.a2 & 2^6 3^3  (-1)^b (3^b- (-1)^\delta 2^a)  \left( 2^{2a} + (-1)^\delta 17 \cdot 2^{a+1} 3^b + 3^{2b}  \right)  \\
72q.3.a3 & -2^6 3^3  (-1)^b (3^b+ (-1)^\delta 2^{a+1})  \left( 2^{2a+5} + (-1)^\delta 2^{a+5} 3^b - 3^{2b}  \right) \\
72q.3.a4 &  2^5 3^3  (-1)^b (3^b+ (-1)^\delta 2^{a-1})  \left( -2^{2a-4} + (-1)^\delta 2^{a+1} 3^b +2 \cdot 3^{2b}  \right)   \\
72q.4.a1 &  -2^5 3^3 d \left( 2 d^2 + 3  \right)  \\
72q.4.a2 &  -2^6 3^3 d \left( d^2 + 12 \right)  \\
72q.4.b1 &  2^5 3^6 d \left(  2 d^2 + 3  \right)  \\
72q.4.b2 &  2^6 3^6 d \left( d^2 + 12 \right) \\
72q.5.a1 &  -2^5 3^3 d \left( (-1)^\delta 2 d^2 + 3 \cdot 2^{a-2} \right)  \\
72q.5.a2 &  -2^6 3^3 d \left( (-1)^\delta d^2 + 3 \cdot 2^{a} \right)  \\
72q.5.b1 &  2^5 3^6 d \left( (-1)^\delta 2 d^2 + 3 \cdot 2^{a-2} \right)  \\
72q.5.b2 &  2^6 3^6 d \left( (-1)^\delta d^2 + 3 \cdot 2^{a} \right) \\
72q.6.a1 & 2^5 3^3 d \left(\frac{d^2+3}{4} \right) \\
72q.6.a2 & 2^6 3^3 d \left( 8d^2+3 \right) \\
72q.6.b1 &  -2^5 3^6 d \left(\frac{d^2+3}{4} \right) \\
72q.6.b2 & -2^6 3^6 d \left( 8d^2+3 \right) \\
72q.7.a1 & 2^5 3^3 d \left(\frac{d^2+3^{b+2}}{4} \right) \\
72q.7.a2 & 2^6 3^3 d \left( 8d^2 + 3^{b+2} \right)  \\
72q.8.a1 & - 2^5 3^3 d \left(2d^2 + 3^{b+2} \right) \\
72q.8.a2 &  -2^6 3^3 d \left( d^2 + 4 \cdot 3^{b+2} \right) \\
72q.9.a1 & 2^6 3^3 d \left( d^2 +(-1)^{\delta_1+\delta_2} 2^{a-3} 3^{b+2} \right) \\
72q.9.a2 & 2^6 3^3 d \left( d^2 + (-1)^{\delta_1+\delta_2} 2^a \cdot 3^{b+2} \right)  \\
72q.10.a1 & 2^6 3^3 d \left( d^2 - 2^{a-3} 3^{b+2} \right)  \\
72q.10.a2 & 2^6 3^3 d \left( d^2 - 2^a \cdot 3^{b+2} \right) \\
72q.11.a1 & 2^6 3^3 d \left( d^2+36 \right) \\
72q.11.a2 & 2^6 3^3 d \left( d^2+288 \right) \\
72q.12.a1 & 2^6 3^3 d \left( d^2+36 \right)  \\
72q.12.a2 & 2^6 3^3 d \left( d^2+288 \right) \\
 \hline
\end{array}
$$



\begin{thebibliography}{1}

\bibitem{BaBe}
M. Bauer and M. A. Bennett,
\newblock Applications of the Hypergeometric Method to the Generalized Ramanujan-Nagell Equation,
\newblock {\em Ramanujan J.} 6 (2002), 209--270.

\bibitem{Be2003}
M. A. Bennett,
\newblock Pillai's conjecture revisited,
\newblock {\em J. Number Th.} 98 (2003), 228--235.

\bibitem{BeGh}
M. A. Bennett and A. Ghadermarzi,
\newblock Mordell's equation : a classical approach,
\newblock {\em LMS J. Comput. Math.} 18 (2015), 633--646.

\bibitem{BLM}
M.A. Bennett, F Luca and J. Mulholland,
\newblock Twisted extensions of the cubic case of Fermat's Last Theorem,
\newblock {\em Ann. Sci. Math. Qu\'ebec} 35 (2011), 1--15.

\bibitem{BMOR}
M.A. Bennett, G. Martin, K. O'Bryant and A. Rechnitzer,
\newblock Explicit bounds for primes in arithmetic progressions,
\newblock in preparation.

\bibitem{Br00}
N.~Bruin.
\newblock On powers as sums of two cubes,
\newblock In {\em Algorithmic Number Theory} (Leiden, 2000),
\newblock Lecture Notes in Comput. Sci., vol 1838, Springer, Berlin, 2000, pp. 169--184.

\bibitem{Bruni}
C. Bruni,
\newblock Twisted Extensions of Fermat's Last Theorem,
\newblock  PhD Thesis, University of British Columbia, 2015, 299 pp.

\bibitem{ChSi}
I. Chen and S. Siksek,
\newblock Perfect powers expressible as sums of two cubes, 
\newblock {\em Journal of Algebra} 322 (2009), 638--656.

\bibitem{Cre}
J. E. Cremona, 
\newblock Algorithms for Modular Elliptic Curves, 2nd ed., Cambridge University Press, 1997.

\bibitem{Dah}
S.~Dahmen.
\newblock Classical and modular methods applied to Diophantine equations,
\newblock PhD thesis, University of Utrecht, 2008.
\newblock Permanently available at\\
\verb|http://igitur-archive.library.uu.nl/dissertations/2008-0820-200949/UUindex.html|

\bibitem{DG}
H. Darmon and A. Granville,
\newblock On the equations $z^m=F(x,y)$ and $Ax^p+By^q=Cz^r$,
\newblock {\em Bull. L.M.S.} 27 (1995), 513--543.

\bibitem{Fr}
N. Freitas,
\newblock  On the Fermat-type equation $x^3+y^3=z^p$, 
\newblock {\em Commentarii Mathematici Helvetici} 91 (2016), 295--304.

\bibitem{FrKr1}
N. Freitas and A. Kraus,
\newblock On symplectic isomorphisms of the $p$-torsion of elliptic curves,
\newblock preprint.

\bibitem{FrKr2}
N. Freitas and A. Kraus,
\newblock On symplectic isomorphisms of the $p$-torsion of elliptic curves II, 
\newblock in preparation.

\bibitem{FNS23n} 
N. Freitas, B. Naskr{\k e}cki and M. Stoll,
\newblock The generalized Fermat equation with exponents $2,3,n$. 
\newblock preprint available at  \url{http://www.math.ubc.ca/~nuno/preprint3.pdf}

\bibitem{HK2002} 
E.\ Halberstadt and A.\ Kraus,
\newblock Courbes de {F}ermat: r\'esultats et probl\`emes,
\newblock {\em J.\ Reine Angew.\ Math.} 548 (2002), 167--234.

\bibitem{HB}
D. R. Heath-Brown,
\newblock Siegel zeros and the least prime in an arithmetic progression,
\newblock {\em Quart. J. Math. Oxford Ser. (2)}, 41 (1990), 405--418.

\bibitem{Kraus1990} 
A. Kraus,
\newblock Sur le d\'efaut de semi-stabilit\'e des courbes elliptiques \`a r\'eduction additive,
\newblock {\em Manuscripta Math.} 69 (1990), no. 4, 353--385.

\bibitem{Kr}
A. Kraus,
\newblock Sur l'\'equation {$a^3+b^3=c^p$},
\newblock {\em Experiment. Math.} 7 (1998), 1--13.

\bibitem{KO}
A. Kraus and J.Oesterle,
\newblock Sur une question de {B}. {M}azur,
\newblock {\em Math. Ann.} 293 (1992), 259--275.

\bibitem{Mu}
J. Mulholland,
\newblock Elliptic Curves with Rational $2$-torsion and Related Ternary Diophantine Equations,
\newblock  PhD Thesis, University of British Columbia, 2006, 340 pp.

\bibitem{Pi}
S. S. Pillai,
\newblock On the equation $2^x-3^y = 2^X+3^Y$,
\newblock {\em Bull. Calcutta Math. Soc.} 37, (1945). 15--20. 
 
 \bibitem{StTi}
 R. J. Stroeker and R. Tijdeman,
 \newblock Diophantine Equations,
 \newblock In: Computational Methods in Number
Theory, Math. Centre Tracts 155, Centr. Math. Comp. Sci., Amsterdam, 1982, 321--369.
 
 \bibitem{Tij}
 R. Tijdeman,
 \newblock On integers with many small prime factors, 
 \newblock {\em Compositio Math.} 26 (1973), 319--330
 
 \bibitem{TiWa}
 R. Tijdeman and L. Wang,
 \newblock Sums of products of powers of given prime numbers,
 \newblock {\em Pacific J. Math.} 132 (1988), 177--193.
 
 \bibitem{Wang}
 L. Wang,
 \newblock Four term equations,
 \newblock {\em Indagationes Math.} 92 (1989), 355--361.
 
\bibitem{Wi}
A. Wiles.
\newblock Modular elliptic curves and Fermat's Last Theorem,
\newblock {\em Ann. of  Math} 141 (1995), 443--551.

 \end{thebibliography}
\end{document}